\documentclass[12pt,reqno]{amsart}

\usepackage[margin=2.4cm]{geometry}
\usepackage[ansinew]{inputenc}
\usepackage[english]{babel}
\usepackage[all]{xy}
\usepackage{amsmath,latexsym,amssymb,verbatim}
\usepackage[usenames]{color}
\usepackage{hyperref}

\newtheorem{thm}{Theorem}[section]
\newtheorem{prop}[thm]{Proposition}

\newtheorem{lem}[thm]{Lemma}
\newtheorem{cor}[thm]{Corollary}
\numberwithin{equation}{section}
\theoremstyle{definition}
\newtheorem{definition}[thm]{Definition}
\newtheorem{remark}[thm]{Remark}
\newtheorem{remarks}[thm]{Remarks}
\newtheorem{ex}[thm]{Example}

\newcommand{\cE}{{\mathcal E}}

\newcommand{\cO}{{\mathcal O}}

\newcommand{\cN}{{\mathcal N}}
\newcommand{\OO}{{\mathcal O}}
\newcommand{\E}{{\mathcal E}}

\newcommand{\M}{{\mathcal M}}

\newcommand{\bbQ}{{\mathbb Q}}
\newcommand{\bbZ}{{\mathbb Z}}
\newcommand{\bbC}{{\mathbb C}}
\newcommand{\bbL}{{\mathbb L}}

\newcommand{\bbA}{{\mathbb A}}
\newcommand{\bbP}{\mathbb P}
\newcommand{\bbR}{{\mathbb R}}

\newcommand{\LL}{{\mathbb L}}

\newcommand{\RR}{\mathbb R}

\DeclareMathOperator{\id}{{id}}

\DeclareMathOperator{\HHom}{\mathcal{H}{\it om}}
\DeclareMathOperator{\HExt}{{\mathcal E}{\it xt}}

\newcommand{\SDer}{{\mathcal{D}er}}
\newcommand{\ext}{{\textstyle{\bigwedge}}}
\newcommand{\exto}{{\bigwedge}}
\newcommand{\punto}{{\scriptscriptstyle \bullet}}
\newcommand{\Dqc}{\text{D}_{\text{qc}}}
\newcommand{\R}{\operatorname{R}}
\newcommand{\iso}{{\,\stackrel {\textstyle\sim}{\to}\,}}
\newcommand{\isom}{\stackrel{\sim}\to}
\newcommand{\Img}{\operatorname{Im}}
\newcommand{\Ker}{\operatorname{Ker}}

\newcommand{\Aut}{\operatorname{Aut}}
\newcommand{\Mat}{\operatorname{Mat}}
\newcommand{\GL}{\operatorname{GL}}
\newcommand{\Ber}{\operatorname{Ber}}
\newcommand{\BBer}{\operatorname{{\mathbb B}er}}
\newcommand{\Hom}{{\operatorname{Hom}}}

\newcommand{\Proj}{\operatorname{Proj}}
\newcommand{\Spec}{\operatorname{Spec}}
\newcommand{\SProj}{\operatorname{\mathbb{P}roj}}
\newcommand{\SSpec}{\operatorname{\mathbb{S}pec}}

\newcommand{\Kos}{\operatorname{Kos}^\bullet}

\newcommand{\SDeRham}{\operatorname{\mathbb{D}eRham}^\bullet}
\newcommand{\Tor}{\operatorname{Tor}}
\newcommand{\hSKos}[1]{\operatorname{\mathbb{K}os}^{#1}}
\newcommand{\SKos}{\operatorname{\mathbb{K}os}^\bullet}
\newcommand{\di}{\operatorname{d}}
\newcommand{\wt}{\widetilde}

\title{On Koszul Complex of a Supermodule}
\author[D. S\'anchez G\'omez]{Dar\'io S\'anchez G\'omez}
\author[F. Sancho de Salas]{Fernando Sancho de Salas}
\address{Departamento de Matem\'aticas and Instituto Universitario de F\'{\i}sica Fundamental y Matem\'aticas
(IUFFyM), Universidad de Salamanca, Plaza de la Merced 1-4, 37008
Salamanca, Spain.}
\email{dario@usal.es, fsancho@usal.es}
\thanks{This work was supported by the Grant PID2021-128665NB-I00 funded by MCIN/AEI/ 10.13039/501100011033 and, as appropriate, by ``ERDF A way of making Europe".
}
\date{\today}

\begin{document}

\maketitle

\begin{center}
\textit{{Dedicated to Professor Daniel Hern\'andez Ruip\'erez on the occasion of  his 70$\,^{th}$ birthday}}
\end{center}

\begin{abstract}
This paper is devoted to an exposition of the Koszul complex of a supermodule and its Berezinian from an intrinsic and as general as possible point of view. As an application, an analogue to Bott's formula in the supercommutative setting computing the cohomology of twisted differential $p$-forms sheaves on the projective superspaces is given.
\end{abstract}



\section*{Introduction}
The Koszul complex appears in different forms across many areas of mathematics. It was originally introduced by J. L. Koszul (\cite{Koszul}) for the purpose of defining a cohomology theory for Lie algebras. Since then, it has turned out to be a useful tool in homological algebra (and in many other contexts including algebraic geometric, commutative algebra and representation theory) that can be used to understand topics like regular sequences, resolutions of modules and the structure of graded algebras.

In these notes we address the problem of studying the universal Koszul complex of a module in the context of the supercommutative algebra. In all honesty, we must point out that it is not the first time this problem has been considered. Some results mentioned in this paper appear, in one guise or another, in Manin's papers (see for instance \cite{Manin88,Manin91}) in \cite{Leites80} and \cite{OP84} and, more recently, in \cite{NR22}. However, unlike those papers, the novelty of our contribution lies in the fact that we provide a purely algebraic approach and we tackle this matter from an intrinsic point of view and in more general terms; no needing to take coordinates (as in \cite{NR22}) in order to make calculations and considering a broad range of superspaces, without limiting ourselves to complex supergeometry (as in \cite{Leites80} or \cite{OP84}), without making extra assumptions about the smoothness of the considered spaces and valid in arbitrary characteristic. Therefore, the aim of this paper is to handle these issues as generally as possible so that the results can be applied to a wide range of cases. 

Following the approach and terminology used in \cite{Sancho00} and \cite{ASS16}, we shall reformulate the universal super Koszul complex of a supermodule in terms of the differential forms of its symmetric superalgebra. Its differential will be given by the inner product with the universal superderivation of the supermodule. This point of view will make clear the relationship between the super Koszul complex and the super De Rham complex which allows us to obtain some results on the acyclicity of the Koszul complex. An unexpected outcome is that the super Koszul complex is not necessarily acyclic in general, even restricting ourselves to the case of free supermodules. Notice further that, whereas on the ordinary commutative setting differential forms anticommute, in the supercommutative setting the differential of odd functions do commute instead. Therefore, the super De Rham complex is not  bounded from above. Consequently, there is no notion of a top differential form and the determinat or canonical module of a free supermodule makes no sense in the supercommutative setting. The Berezinian (\cite{HRM,Manin88}) provides a substitute for the notion of the determinant when passing from the commutative setting to the supercommutative setting. In this paper, as a first application of this way of describing the Koszul complex, we shall show that the Berezinian of a free supermodule can be understood as a dualizing sheaf, leading to the computation of the Berezinian of a supermatrix in a straightforward way. As a second application, we shall give a formula (which we have called the super Bott formula) to compute the cohomology of the twisted sheaves of differential $p$-forms on the projective superspace which reduces to the Bott formula (\cite{Bott57}\cite{OSS80}) in the ordinary case.

Before explaining how the paper is structured, a remark is now in order. Although throughout the paper we are considering only the affine case, that is, supermodules over a superring, both assertions and proofs carry over completely to the global case considering locally free sheaves of supermodules over superschemes.

The structure of this paper is as follows. Seeking to provide a completely self-contained exposition we begin, in Section \ref{sec:generalities}, with a brief review of basic notions that appear in commutative superalgebra. We focus our attention on the symmetric algebra $S^\punto_A M$ and the exterior algebra $\bigwedge^\punto_A M$ associated to an $A$-supermodule $M$, and also on the module of K\"ahler differentials $\Omega_{B/A}$ of an $A$-superalgebra $B$, since these will be the building blocks in defining the Koszul complex of a supermodule. The section ends recalling the notion of superscheme and making contact with the relative dualizing complex associated to a closed immersion of superschemes. In particular, we establish two of its remarkable properties, transitivity (Proposition \ref{transitivity}) and flat base change (Proposition \ref{flatbasechange}), which will be needed in later sections.

In Section \ref{sec:DeRahmKos}, the universal Koszul complex of a module in the supercommutative setting is defined following the approach and terminology introduced in \cite{Sancho00}. For any $A$-supermodule $M$, its super Koszul complex $\SKos_A(M)=\bigwedge^\punto_A M\otimes_A S^\punto_AM$ is a graded complex whose $n$-th graded component $\SKos_A (M)_n$ is the complex
\begin{equation*}
\xymatrix@C=17pt{
0\ar[r]& \bigwedge^{n}_A M\ar[r]&\cdots\ar[r]&  \bigwedge^{p}_A M\otimes_{A} S_{A}^{n-p}M\ar[r]&\cdots\ar[r]  &S_{A}^n M\ar[r] & 0
}
\end{equation*}
We firstly show  the existence of a canonical isomorphism $\Omega_{S^\punto_A M/A}\iso M\otimes_A S^\punto_A M$ of graded $S^\punto_A M$-supermodules (Theorem \ref{thm:relativediff}). This allows to understand the super Koszul complex of $M$ as the complex $(\Omega^{\bullet}_{S^\punto_A M/A}, i_D)$ of graded $S^\punto_A M$-supermodules of differential forms, whose differential $i_D\colon \Omega^{p}_{S^\punto_A M/A}\to \Omega^{p-1}_{S^\punto_A M/A}$ is the inner contraction given by the superderivation $D\colon S^\punto_A M\to S^\punto_A M$ consisting in multiplication by $n$ on $S^n_A M$. Expressing the super Koszul complex in terms of differential forms has the advantage of highlighting its connection with the super De Rham complex. Namely, both differentials are linked through Cartan's formula, $i_D\circ d+d\circ i_D=n\, Id$ on $\SKos_A(M)_n$, resulting in an acyclicity outcome for the non negative homogeneous components of the super Koszul complex (Theorem \ref{thm:aciclicity}). This also will yield a splitting result (Corollary \ref{cor:Bott}) which will be essential for further cohomological computations in Section \ref{sec:Bott} as we shall explain later on. It should be noted that the characteristic singularities of the supercommutative context, due to the presence of odd variables, make some unexpected facts arise. For instance, the super Koszul complex is unbounded below and, more remarkably, even for free supermodules the super Koszul complex may fail to be acyclic (Theorem \ref{thm:Koszulfreemodule}).

Section \ref{sec:Berezinians} deals with the Berezinian of a supermodule which is defined as the dual of its super Koszul complex. It is fair to say that results regarding  the Berezinian have previously appeared (see for instance \cite{OP84} and \cite{NR22}). In this paper, however, we provide a completely general and detailed treatment of this fundamental construction in commutative superalgebra from an intrinsic approach and in a fairly general way. We compute the cohomology modules of the Berezinian complex of a free supermodule (Theorem \ref{thm:CohomBerFree}) and we define the Berezinian module of a free supermodule $L$  of rank $(p|q)$ as the $p$-th cohomology module of the dual of the super Koszul complex of $L$. We also prove (Theorem \ref{berziniano=dualizante}) that the Berezinian module of a free $A$-supermodule is the dualizing sheaf corresponding to the closed immersion given by the zero section $\SSpec A\hookrightarrow  \SSpec S^\punto_A L$. Furthermore, we show that the Berezinian of the automorphism of a free supermodule, which plays the role of the determinant of an automorphism of a purely even free module, is an isomorphism between dualizing sheaves. As a consequence, we recover in a natural way (Proposition \ref{prop:Berezin}) the supercommutative analogue of the classical determinant of a matrix, the so-called Berezin function or Berezin determinat (\cite{Ber87, Leites80}): $$\Ber \left ( 
\begin{array}{c|c}
A & B \\
\hline 
C & D
\end{array}
\right )= \det(D^{-1})\det(A-BD^{-1}C)\,.$$
Finally, in Section \ref{sec:Bott} we provide an explicit formula computing completely the cohomology of the sheaves $\Omega^p_{\bbP^{m|n}}(r)$ on the projective superspace $\bbP^{m|n}$ (we use the standard notation $\Omega^p_{\bbP^{m|n}}(r)=\Omega^p_{\bbP^{m|n}}\otimes\cO_{\bbP^{m|n}}(r))$. This result reduces to Bott's formula (\cite{Bott57}) to the ordinary projective space setting $n=0$. We follow the approach presented in \cite{ASS16} where (in a joint work with B. Andreas) we generalized Bott's formula to projective bundles by using the universal Koszul complex of a module. At first, one might think that the super Bott formula obtained in this paper is just a straightforward generalization and that the proofs go basically parallel to those given in \cite{ASS16}. However, and as we will show, the analogues of the key results used in the commutative case do not extend to the supercommutative case. This leads us to employ a different strategy, definitively turning our result into a highly non-trivial generalization. Firstly, given a free supermodule $L=A^{m|n}$, we define a complex $\wt{\SKos_A}(L)=(\widetilde{\Omega}^\punto_{B/A}, i_D)$ of supermodules on $\bbP^{m|n}$ taking homogeneous localization in the super Koszul complex of $L$ (here $B$ denotes $S^\punto_AL$). We prove that this complex on $\bbP^{m|n}$ is acyclic and its factors (cycles or boundaries) are the sheaves $\Omega^p_{\bbP^{m|n}}$ (Proposition \ref{prop:Koszultilde}). This provides an exact sequence
$$0\to\Omega^p_{{\bbP}/A}\xrightarrow{} \widetilde{\Omega}^p_{B/A}\xrightarrow{} \widetilde{\Omega}^{p-1}_{B/A}\xrightarrow{}\cdots\xrightarrow{} \widetilde{\Omega}_{B/A}\xrightarrow{}  \cO_{\bbP}\to 0$$
which, in contrast with the ordinary case, is not necessarily $\Gamma$-acyclic (Corollary \ref{cor:cohhomdif}). In particular, the above complex $\{\widetilde{\Omega}^i_{B/A}\}_{i=0}^p$ does not provide a suitable resolution of $\Omega^p_{\bbP/A}$ to compute its cohomology modules. Then, the approach of \cite{ASS16} to obtain Bott's formula in the purely even setting does work in the supercommutative context. Instead, we shall compute the cohomologies $H^i(U,\pi^*\Omega^p_{{\bbP}/A})$ by means of the super Koszul complex (Proposition \ref{prop:Bott}), considering the affine map 
$\pi\colon U=\bbA^{m+1|n}-\{0\}\to \bbP^{m|n}$ given by the action of the multiplicative group, and we obtain the super Bott formula for the projective superspace as a consequence (Theorem \ref{thm:Bott} and Corollary \ref{cor:Bott}).

\section{Generalities on supercomutative algebra}\label{sec:generalities}
To start with, we will give a brief overview of some basic concepts and facts of supercommutative algebra which, in the sequel, will be useful for both the convenience of the reader unfamiliar with this subject and to provide the stage and the algebraic machinery to deal with the Koszul complex of a supermodule. For a thorough exposition of supercommutative algebra we suggest to visit classical references such as \cite{Leites80, Manin88, Manin91,BBR91,Witt84} or more recent papers as \cite{CCF11,BR2023}

\begin{definition}
A superring $A$ is a $\bbZ_2$-graded ring, that is, $A=A_0\oplus A_1$ where each $A_i$ is an abelian additive subgroup of $A$ and the product satisfies $A_iA_j\subseteq A_{i+j}$ (index modulo 2). 
\end{definition}
The elements of $A_0$ are called homogeneous even elements whereas elements of $A_1$ are called homogeneous odd elements. If $a\in A_i$, then we put $|a|=i$ and $|a|$ is called the parity of $a$.  A morphism of superrings is a parity preserving morphism of rings. 

A superring $A=A_0\oplus A_1$ is said to be commutative if the Quiller-Koszul signs rule holds:
\begin{equation}\label{signsrule}
\begin{aligned}
&ab=ba, \text{ either if } a,b\in A_0, \text{ or if } a\in A_0 \text{ and } b\in A_1\\
&a^2=0 \text{ if } a\in A_1.
\end{aligned}
\end{equation}
Note that these two conditions are equivalent to 
$ab=(-1)^{|a|\,|b|} ba$ for any two homogeneous elements, provided $2$ is not zero divisor in $A$.

To avoid confusion, commutative superrings are also called supercommutative rings.
\subsection{Supermodules}

Fixed a superring $A$, a left $A$-supermodule is a $\bbZ_2$-graded abelian group $M$ which is a left $A$-module in the usual sense and, in addition, for homogeneous elements $a_i\in A$ and $m_j\in M$ the product $a_i\cdot m_j$ lies in $M_{i+j}$. Analogously, a right supermodule is a right module in the usual sense in such a way that the product (say $m*a$) respects the $\bbZ_2$-grading. When $A$ is a supercommutative ring, any left $A$-module $M$ also has a natural right $A$-module structure by setting 
\begin{equation}\label{eq:LRmodule}
m* a=(-1)^{|a|\,|m|} a\cdot m
\end{equation}
for all homogeneous elements $a\in A$ and $m\in M$, and then extending by linearity, in such a way that both actions commute $(a\cdot m)*b=  a\cdot (m* b)$ for all $a,b\in A$ and $m\in M$. Hence, $M$ can be considered both sided supermodule over the supercommutative ring $A$ so that $M$ is just called an $A$-supermodule without mentioning right of left. When no confusion is clear, we will use just juxtaposition denoting both actions.

For a right $A$-supermodule $M$, we denote by $\Pi M$ the parity-swapping $A$-supermodule, that is, the shifted right $A$-supermodule $M[1]$ where $M[1]_0=M_1$ and $M[1]_1=M_0$. The right action of $A$ on $\Pi M$ is defined as the right action of $A$ on $M$. Note that the corresponding compatible left module structures just described above (Equation \ref{eq:LRmodule}) on $M$ and $M[1]$ differ on a sign depending on the parity of $a\in A$.

Morphisms of $A$-supermodules (also called even morphisms) are defined to be parity preserving maps that commute with the action of $A$. Since a morphism $f$ preserves parity, when $f$ commutes with the right action of A it also commutes with the left action. This set of morphisms does not carry any natural supermodule structure itself, we need to extend it by adding parity reversing morphisms. An odd morphism $f\colon M\to N$ is a parity reversing morphism which is $A$-linear with respect the right action of A, $f(ma)=f(m)a$. That is, an even morphism $f\colon M\to N[1]$ where the right action of $A$ in $N[1]$ is the one induced by the right action of $A$ in $N$. Therefore the set ${\HHom}_A(M,N)$ of all $A$-linear morphisms $M\to N$ is endowed with a natural $\bbZ_2$-grading and it is a $A$-supermodule with the following action of A:
$$(a\cdot f)(m)=a\cdot(f(m))\,.$$

\begin{definition}
An $A$-supermodule $M$ is called a free supermodule if $M$ is free as $A$-module with a basis consisting of homogeneous elements.
\end{definition} 
This is equivalent to say that there exists an isomorphism of $A$-supermodules 
$$\bigoplus_{i\in I} \Pi ^{d_i}A\isom M$$
given as follows: consider the free $A$-supermodule $\bigoplus_{i\in I}\Pi^{d_i} A$, with $I$ an indexing set and $d_i\in \bbZ_2$. For each $i\in I$ the element $e_i$ consisting in $1$ in the $i$-th component and zero elsewhere is homogeneous of parity $|e_i|=d_i$. Therefore, given  $\{m_i\}_{i\in I}$ a homogeneous basis of $M$, where $|m_i|=d_i$, one has the above $\bbZ_2$-graded isomorphism which maps $e_i$ to $m_i$.	It is customary to enumerate the elements of a basis in a free module so that the even ones come first and then the odd ones. For simplicity, a free $A$-supermodule $M$ with homogeneous basis $\big\{\{m_i\}_{i\in I_0},\, \{\eta_j\}_{j\in I_1}\big\}$, being $|m_i|=0$ and $|\eta_j|=1$, shall be denoted $M\simeq A^{I_0}\oplus\Pi A^{I_1}$. If $|I_0|=p$ and $|I_1|=q$ the free $A$-supermodule $A^{p}\oplus\Pi A^{q}$ is also denoted as $A^{p|q}$ and we say that has rank $p|q$.

\begin{definition}
Let $A$ be a supercommutative ring. Given two $A$-supermodules $M$ and $N$, the usual tensor product $M\otimes_A N$ (considering $M$ as a right $A$-module and $N$ as left $A$-module) is a $A$-supermodule with the right action $$(m\otimes n)\cdot a=m\otimes n\cdot a$$ and is  $\bbZ_2$-graded by setting
$$(M\otimes_A N)_n=\bigoplus_{i+j=n}(M_i\otimes N_j)\, \text{ (index modulo 2). }$$ 
\end{definition}

\subsection{Superalgebras}

Fix a supercommutative ring $A$.  An $A$-superalgebra $B$ is defined to be an $A$-algebra with a  $\bbZ_2$-grading such that $A_iB_j\subset B_{i+j}$ (index modulo 2), that is, an $A$-superring $B$ endowed with a morphism $A\to B$ of superrings. 

We say that a superring (resp. superalgebra) $A$ is a graded superring (resp. superalgebra) if $A$ is a $\bbZ$-graded ring (resp. algebra) and both grading are compatible, that is, if $A_i$ denotes the homogeneous component of degree $i$ with respect to the $\bbZ$-grading and $A_+$ and $A_-$ denote respectively the even and odd elements of $A$, then  $(A_i)_+=A_+\cap A_i$ and $(A_i)_-=A_i\cap A_-$, for all $i\in\bbZ$.
\begin{ex}
Let $A$ be a  superring. The polynomial superalgebra over $A$ in $m$ even variables $x_i$ and $n$ odd variables $\theta_i$ is defined as the algebra over $A$ generated by $x_i$ and $\theta_i$ under the following relations: $x_ix_j=x_jx_i,\, x_i\theta_j=\theta_jx_i$ and $\theta_i\theta_j=-\theta_j\theta_i$. It is denoted by $A[x_1,\dots,x_m|\theta_1,\dots,\theta_n]$ and is a graded superalgebra where all variables are free of degree $1$.
\end{ex}

If $M$ is an  $A$-supermodule and $B$ is a graded $A$-superalgebra, then $M\otimes_{A} B$ is a graded $B$-supermodule with $\bbZ$-grading $(M\otimes_{A} B)_n=M\otimes_{A} B_n$.

\begin{ex}
Let $A$ be a ring and let $M$ be a free $A$-module with basis $\{\theta_1,\dots,\theta_n\}$. The exterior algebra $\bigwedge^\punto_A M$ is a graded $A$-superalgebra. The $k$-th homogeneous component with respect the $\bbZ$-grading is $\bigwedge^k_A M$. We put a $\bbZ_2$-grading by saying that even elements are $(\bigwedge ^\punto _A M)_0=\bigwedge^{2i}_A M$ and odd elements are $(\bigwedge ^\punto _A M)_1=\bigoplus_i \bigwedge ^{2i+1} _A M$, that is, the parity $|\omega|$ of $\omega$ equals the degree of $\omega$ mod $\bbZ_2$. The tensor product induces a multiplication law in $\bigwedge^\punto_A M=\bigoplus_k\bigwedge^k_A M$ which is denoted by $\wedge$ and satisfies $\omega_p\wedge \omega_q=(-1)^{pq}\omega_q\wedge\omega_p$. Since $(-1)^{|\omega|}=(-1)^{\deg(\omega)}$ one gets that $\bigwedge^\punto_A M$ is a supercommutative $A$-algebra.
\end{ex}
In order to define the symmetric and exterior algebras of a supermodule over a superring we have to be a bit more careful. Let $M$ be a $A$-supermodule and let $T^\punto_A(M)=\displaystyle\bigoplus_{i\geq 0}M^{\otimes_A^i}$ be its tensor superalgebra. 
\begin{definition}
The symmetric superalgebra over $M$ is defined as $S^\punto_A M=T(M)/I_M$, where $I_M$ denotes the $\bbZ_2$-graded ideal generated by the elements of $T^2(M)$ of the form $m\otimes m$ for all homogeneous odd element $m\in M$ and elements of the form $m\otimes m'-(-1)^{|m||m|'}m'\otimes m$, where $m$ and $m'$ run over all homogeneous elements of $M$. Then $S^\punto_A M=\bigoplus_{n\geq 0} S^n_A M$ and each homogeneous component $S^n_A M$ is an $A$-supermodule.
The tensor product in $T^\punto_A(M)$ gives rise to a product in $S^\punto_A M$ which satisfies the Quiller-Koszul signs rule (\ref{signsrule}). Therefore $S^\punto_A M$ is a graded supercommutative $A$-algebra, where the  structure of $\bbZ$-graded $A$-algebra is that induced by the symmetric product and that of $A$-superalgebra is that inherited by the parity of $M$. 
\end{definition}

Similarly, the exterior algebra $\bigwedge^\punto_A M$ of a supermodule $M$ is defined in the following way:

\begin{definition}
The exterior algebra $\bigwedge^\punto_A M$ is defined to be the quotient of $T^\punto_A(M)$ by the ideal generated by elements of the form $m\otimes m'+(-1)^{|m||m'|}m'\otimes m$, for all homogeneous elements $m,m'\in M$ and elements of the form $m\otimes m$ for all homogeneous even element $m\in M$.  It follows that $\bigwedge^\punto_A M=\bigoplus_{n\geq n}\bigwedge^n_A M$ and the $n$-th exterior power  $\bigwedge^n_A M$ of $M$ is an $A$-supermodule. The product of $T^\punto_A(M)$ induces one in $\bigwedge^\punto_A M$ which is denoted by $\wedge$ and satisfies the following relationship: 
$$\omega_p\wedge\omega_q=(-1)^{pq}(-1)^{|\omega_p||\omega_q|}\omega_q\wedge\omega_p$$
for any $\bbZ_2$-homogeneous elements $\omega_p\in \bigwedge^p_A M,\, \omega_q\in \bigwedge^q_A M$. Notice that $\bigwedge ^n_A M$ is a non commutative $A$-superalgebra.
\end{definition}

It is worth noticing that, if $A$ is purely even and $M\simeq A^p$, then  $S^\punto_A(M)=\bigwedge ^\punto _A (\Pi M)=A[x_1,\cdots x_p]$ is purely even  and $S^\punto_A(\Pi M)=\bigwedge ^\punto _A(M)$ is a $A$-supercommutative algebra.  However, if $A$ is an arbitrary supercommutative ring and $M\simeq A^p$ is a free $A$-supermodule of rank $(p|0)$, then 
\begin{equation}\label{eq:simetricotwisted}
S^k_A(\Pi M)\simeq \Pi^k{\bigwedge}^k_A M \text{ and } {\bigwedge}^k_A (\Pi M)\simeq \Pi^kS^k_A M.
\end{equation}

As in the ordinary case (that is, the purely even case) one proves:
\begin{prop}\label{prop:SEproperties}
\hspace{2em}
\begin{enumerate}
\item Base change: Let $A\to A'$ be a morphism of superrings and let $M$ be an $A$-supermodule. Then $S^\punto_{A'}(M\otimes_A A')\simeq S^\punto_A M\otimes _A A'$ and $\bigwedge^\punto_{A'}(M\otimes_A A')\simeq \bigwedge^\punto_A M\otimes _A A'$ .
\item Additivity: Let $M$ and $N$ be two $A$-supermodules. Then $S^\punto_A(M\oplus N)\simeq S^\punto_AM\otimes_A S^\punto_A N$ and $\bigwedge^\punto_A(M\oplus N)\simeq \bigwedge^\punto_AM\otimes_A \bigwedge^\punto_A N$. 
\end{enumerate}
\end{prop}

\begin{definition}
Let $B$ be an $A$-superalgebra, we say that $B$ is finitely generated (as $A$-superalgebra), if there is a finite number of homogeneous elements $\xi_1,\dots, \xi_m, \eta_1,\dots, \eta_n$ in $B$ such that each element of $B$ can be written as a polynomial in the $\xi_i$ and $\eta_i$ with coefficients in $A$. In other words, there are even elements $\xi_1,\dots, \xi_m$ and odd elements $\eta_1,\dots,\eta_n$ in $B$ and a surjective morphism of superrings from the polynomial $A$-supercommutative algebra $A[x_1,\dots, x_m\,|\, \theta_1,\dots, \theta_n]$ to $B$
which maps $x_i$ to $\xi_i$ and $\theta_i$ to $\eta_i$. So that, any finitely generated superalgebra over $A$ is a quotient of $A[x_1,\dots, x_m]\otimes_A \exto^\punto_A E_\theta$ by some homogeneous ideal, where $\exto^\punto_A E_\theta$ denotes the Grassmann algebra generated by $\{\theta_1,\dots,\theta_n\}$.
\end{definition}

\subsection{Kahler differentials and De Rham complex}
Let $B$ be an $A$-superalgebra and $M$ a $B$-supermodule, a superderivation of degree $0$ (or an even derivation) of $B$ over $M$ is a morphism of $A$-supermodules $D\colon B\to M$ obeying the Leibniz rule 
$$D(bb')=Dbb'+bDb'\,.$$

A superderivation of degree $1$ (or an odd superderivation) of $B$ over $M$ is an anti-derivation from $B$ to $M$ of degree $1$, that is, a morphism of $A$-modules $D\colon B\to M$ which changes the parity of elements and satisfies the Leibniz rule in the form:
$$D(bb')=Dbb'+ (-1)^{|b|}bDb'\,.$$
Since the left module structures on $M$ and $M[1]$ differ on a sign depending on the parity of $b\in B$, the expression $D(bb')=Dbb'+ (-1)^{|b|}bDb'$ reads as $D(bb')=Dbb'+ bDb'$ on $M[1]$. In other words, a superderivation of degree $1$ (or an odd superderivation) is a superderivation of degree $0$ from $B$ to $M[1]$.

By a superderivation we mean a sum of even and odd superderivations. We write $\SDer_A(B,M)$ for the set of all $A$-superderivations of $B$ into a $B$-supermodule $M$. The sum of two superderivations is a derivation and if $D$ is a superderivation and $b\in B$, then $bD$ defined as $(bD)(b')=bD(b')$ is also a superderivation. Then $\SDer_A(B,M)$ is a $B$-supermodule. 

\begin{ex}\label{ex:Eulerfield}
Let $M$ be a $A$-supermodule and denote $B=S^\punto_A M$. The map $D\colon B\to B$ consisting in multiplication by $d$ on $B_d=S^d_A M$ is an (even) $A$-superderivation that we shall call the homothety field. In particular, if $M$ is a free $A$-supermodule of rank $m|n$, then $B=A[x_1,\dots,x_m | \theta_1,\dots\theta_n]$ and $D$ is the so-called Euler vector field 
$$D=\sum_{i=1}^m x_i\frac{\partial}{\partial x_i}+\sum_{i=1}^n \theta_i\frac{\partial}{\partial\theta_i}\,.$$
with $|\frac{\partial}{\partial x_i}|=0$ and $|\frac{\partial}{\partial \theta_i}|=1$.
\end{ex}

The tensor product $B\otimes_A B'$ of two $A$-superalgebras is defined as the usual tensor product of $A$-modules endowed with the product
$$(b_1\otimes b_2)\cdot (b_1'\otimes b_2')=(-1)^{|b_2||b_1'|}(b_1\cdot b_1')\otimes (b_2\cdot b_2')\,.$$

Let $B$ be an $A$-superalgebra. The natural morphism $$B\otimes_A B\to B,\,\, b\otimes b'\mapsto bb'$$ is an (even) morphism of superalgebras. The kernel $\Delta$ is a $\bbZ_2$-graded ideal of $B\otimes_A B$ and a right  $B=(A\otimes_A B)$-supermodule generated by all elements of the form $b\otimes 1-1\otimes b$, where $b$ runs over all elements of $B$. The quotient $\Delta/\Delta^2$ is called the module of K\"ahler differentials of $B$ relative to $A$ and is denoted by $\Omega_{B/A}$. The canonical map 
$$\di\colon B\to\Omega_{B/A}$$
given by $\di b=\overline{b\otimes 1-1\otimes b}$ is called the differential of $B$ over $A$  or the canonical derivation.

$\Omega_{B/A}$ is a right $B$-supermodule generated by all elements of the form $\di b$, for all $b\in B$, and there exists an isomorphism of right $B$-supermodules
\begin{equation}\label{eq:Der}
\HHom_B(\Omega_{B/A},M)\simeq \SDer_A(B,M),\quad f\mapsto f\circ\di\,.
\end{equation}

Let us denote by $\Omega^p_{B/A}$ the $p$-th exterior power of $\Omega_{B/A}$. For each $p$ the differential $\di\colon B\to\Omega_{B/A}$ lifts (in a unique way) to the exterior differential $\di\colon \Omega^p_{B/A}\to\Omega^{p+1}_{B/A}$, which is a nilpotent morphism of $A$-supermodules satisfying 
$$\di(\omega_p\wedge\omega_q)=\di\omega_p\wedge \omega_q+(-1)^{p} \omega_p\wedge\di \omega_q\,.$$ 
In other words, $d$ is the only $A$-antiderivation of degree $1$ on $\Omega^\bullet_{B/A}=\bigwedge^\punto_{B} \Omega_{B/A}$ which in degree zero coincides with the differential and verifies $d^2=0$. 

\begin{definition}
The super De Rham complex of $B$ over $A$ is defined to be the complex of $A$-supermodules  $(\Omega^\bullet_{B/A}, d)$.
\end{definition}

The following assertions hold.
\begin{prop}\label{prop:Diffproperties}
Let $B$ and $B'$ be supercommutative $A$-algebras. Then
\begin{enumerate}
\item Base change: $\Omega_{B\otimes_A B'/B'}\simeq \Omega_{B/A}\otimes_A B'$ and hence there is an isomorphism of complexes
$$\Omega^\bullet_{B\otimes_A B'/B'}\simeq \Omega^\bullet_{B/A}\otimes_A B'\,.$$
\item Additivity: $\Omega_{B\otimes_A B'/A}\simeq \Omega_{B/A}\otimes_A B'\oplus \Omega_{B'/A}\otimes_A B$ and hence there is an isomorphism of complexes
$$\Omega^\bullet_{B\otimes_A B'/A}\simeq \Omega^\bullet_{B/A}\otimes_A\Omega^\bullet_{B'/A}\,.$$
\end{enumerate}
\end{prop}

Given an $A$-superderivation $D\colon B\to B$ let us denote by $i_D$ the morphism of $B$-supermodules $i_D\colon\Omega_{B/A}\to B$ given by Equation \ref{eq:Der}. This morphism is called the interior product with $D$ (also the contraction with $D$) and, for any $p\geq 0$, extends in a unique way to a morphism of $B$-supermodules 
$$i_D\colon \Omega^{p+1}_{B/A}\to\Omega^p_{B/A}$$
obeying the graded Leibniz rule

$$i_D(\omega_p\wedge\omega_q)=(i_D\omega_p)\wedge\omega_q+(-1)^p\omega_p\wedge i_D(\omega_q)\,.$$
That is, $i_D$ is the unique $B$-antiderivation of degree $-1$ on the exterior algebra $\Omega^\bullet_{B/A}=\ext^\punto_{B} \Omega_{B/A}$ such that $i_D(db)=Db$. Furthermore, by induction on $p$ is straightforward to show that $i_D\circ i_D=0$. Then, we have a complex $(\Omega^\bullet_{B/A}, i_D)$ of $B$-supermodules.

\begin{definition}[Cartan's formula]
Let $D\colon B\to B$ be an $A$-superderivation. The Lie derivative with respect to $D$ is defined to be the anticommutator of the exterior differential and the interior product with $D$. It is denoted by $D^L$. That is,
\begin{equation}
D^L=i_D\circ \di+\di\circ i_D\,.
\end{equation} 
\end{definition}
Notice that the Lie derivative is the unique $A$-derivation of degree 0 on $\Omega^\bullet_{B/A}$ such that $D^L(b)=Db$ and $D^L(db)=d(Db)$, for all $b\in B$.

\begin{ex}\label{ex:Lie}
The Lie derivative $D^L=i_D\circ\di+\di\circ i_D\colon \Omega^p_{B/A}\to \Omega^p_{B/S}$ corresponding to the homothety field (Example \ref{ex:Eulerfield}) is the multiplication by $n$ on the homogeneous component of $\Omega^p_{B/A}$ of degree $n$, since it is so over $b_n$ and $\di b_n$ for any $b_n\in B$ of degree $n$.
\end{ex}
\subsection{Superschemes and the relative dualizing complex on superschemes}
In order to make the paper as self-consistent as possible, let us recall the notion of superscheme (we shall omit some details, for which we refer to \cite{BR2023}). A locally ringed superspace is a topological space $X$ endowed with a sheaf $\cO_X$ of supercommutative rings such that all stalks $\cO_{X,x}$ are local superrings. A morphism $(X,\cO_X)\to (Y,\cO_Y)$ of locally ringed superspaces consists of a continuous map $f\colon X\to Y$ together with a homogeneous morphism of sheaves of supercommutative rings $f^\#\colon\cO_{{Y}}\to f_*\cO_{{X}}$, such that for any $x\in X$ the induced morphism of local superrings $\cO_{{Y},f(x)}\to \cO_{{X},x}$ is local.

As a fundamental example of locally ringed superspace is the superspectrum of a superring, which is defined as follows. Given a supercommutative ring $A$, we denote by $\bar{A}$ the so-called bosonic reduction (or the body) of $A$, that is, $\bar{A}=A/J$ where $J=A_1^2\oplus A_1$ is the ideal generated by the odd elements. If we consider $A_0$ as a commutative ring, then $A^2_1$ is an ideal of $A_0$ contained in all prime ideals of $A_0$ and the natural projection $A_0\to \bar{A}\simeq A_0/A_1^2$ induces an homeomorphism of topological spaces $\Spec \bar{A}\isom \Spec A_0$ ($\bar{A}$ and $A_0$ differ only by nilpotent elements). Topologically, one defines the superspectrum of $A$ as the spectrum of the bosonic part of $A$. However, with respect to the algebraic structure, the structure sheaf (the localization sheaf) is a sheaf of supercommutative rings defined as follows:  consider $X=\Spec\,\bar{A}$ and for any non-nilpotent even element $f\in A$ set $U_f=X-(f)_0$ and let $A_f$ be the localization of $A$ at the multiplicative system defined by $f$. One gets a sheaf, say $\widetilde{A}$, of supercommutaive rings on $X$ by letting $\widetilde{A}(U_f)=A_f$.

\begin{definition}
An affine superscheme is a pair of the form $\SSpec A=(\Spec \bar{A}, \wt{A})$ where $A$ is a supercommutative ring and $\wt{A}$ its localization sheaf defined above. A locally ringed superspace is called a superscheme if it is locally isomorphic to an affine superscheme. As convention, all the superschemes are supposed to be locally noetherian. Morphisms of superschemes are just morphisms of locally ringed superspaces.
\end{definition}
\begin{ex}
Let $B(m,n)=A[x_1,\cdots,x_{m},\theta_1,\dots,\theta_n]$ be the $A$-supercommutative algebra of polynomials in $m$ even variables $x_i$ and $n$ odd variables $\theta_i$. The affine $(m,n)$-superspace over $A$  is defined as the superspectrum of the superring $B(m,n)$, that is, the superscheme
$$\bbA^{m|n}_A=(\Spec\,\bar{A}[x_1,\dots,x_{m}],\widetilde{B(m,n)})\,.$$
\end{ex}
\begin{ex}
Let $B(m,n)=A[x_0,\dots,x_{m},\theta_1,\dots,\theta_n]$ and 
put $\bbP_{\bar{A}}^m=\Proj\,{\bar{A}}[x_0,\dots,x_{m}]$. Define $\widetilde{B(m,n)}^h$ to be the sheaf of homogeneous localization, that is, the sheaf of superrings whose sections on each basic open subset $U_{x_i}^h=\bbP_{\bar{A}}^m-\big(x_i\big)^h_0$ are given by 
$$
\begin{aligned}
\widetilde{B(m,n)}^h(U_{x_i}^h)&=\big[B(m,n)_{x_i}\big]_0=\\&=\Big\{ \frac{P_d(x_0,\dots,x_{m},\theta_1,\dots,\theta_n)}{x_i^d}\colon P_d(x_j|\theta_k) \text{ is } \bbZ\text{-homogeneous of degree }d\Big\}=\\
&={A[\frac{x_0}{x_i},\dots,\widehat{\frac{x_i}{x_i}},\dots,\frac{x_m}{x_i},\frac{\theta_1}{x_i},\dots\frac{\theta_n}{x_i}]}\,. 
\end{aligned}
$$
That is, if we put $B=B(m,n)$, then on each basic open subset $U_{x_i}$ we  have
$${\widetilde{B}^h\,}_{|_{U_{x_i}}}=\wt{[B_{x_i}]_0}\,.$$
The projective $(m,n)$-superspace over $A$, denoted by $\bbP^{m|n}_A$, is defined as the projective superspectrum of $B(m,n)$, that is, the superscheme $$\SProj_AB(m,n)=(\Proj\,{\bar{A}}[x_0,\dots,x_{m}], \widetilde{B(m,n)}^h)\,.$$ 
\end{ex}

Since $J$ consists of nilpotent elements, the underlying topological space of $\bbA^{m|0}_A$ is homeomorphic to $\Spec A[x_0,\cdots,x_{m}]$, so that, $\bbA^{m|0}_A$ and $\bbP^{m|0}_A$ are the ordinary affine space $\bbA^{m}$ and the ordinary projective space $\bbP^m$, respectively. Moreover, $\bbA^{m|n}_A =\bbA^{m|0}\times_{\SSpec A}\bbA^{0|n}$. For simplicity, we shall omit the subcript $A$. 

The closed immersion of affine superschemes $s_0\colon \SSpec\, A\hookrightarrow \bbA^{m+1|n}$ induced by the exact sequence
$$0\to (x_0,\dots,x_m,\theta_1\dots,\theta_n)\to A[x_0,\cdots,x_{m},\theta_1,\cdots,\theta_n]\to A\to 0$$ 
is called the zero section of $\bbA^{m+1|n}\to  \SSpec\, A$. Denoting its image by $\{0\}$ and by $\bbA^{m+1|n}-\{0\}$ the corresponding complementary open subset, there exists a natural projection (given by the action of the multiplicative group) 
\begin{equation}\label{eq:affinequotient}
\pi\colon \bbA^{m+1|n}-\{0\}\to \bbP^{m|n}
\end{equation}

In general, given an $A$-module $M$, the image of the natural map $M\otimes_A S^\punto_A M\xrightarrow{\theta} S^\punto_AM$ is the ideal generated by $M$, that is, $I=\bigoplus_{i>0}S^i_AM$ and there exists an exact sequence
$$0\to I\to  S^\punto_A M\xrightarrow{s_0} S^\punto_A M/I\isom A\to 0$$
For every $A$-supermodule $M$ let us denote by $\mathbb{V}(M)=\Spec(S^\punto_A M)$ the linear superscheme associated to $M$.
\begin{definition}
Let $M$ be an $A$-supermodule. The linear superscheme associated to $M$ is the affine superscheme
$$\mathbb{V}(M)=\SSpec(S^\punto_A M)\xrightarrow{\pi}{\SSpec A}$$
where the morphism $\pi\colon\mathbb{V}(M)\to \SSpec A$ is the natural one induced by the inclusion $A\hookrightarrow B$. The zero section of $\pi\colon\mathbb{V}(M)\to \SSpec A$ is defined as the closed immersion $$\sigma_0\colon (I)_0=\SSpec A\hookrightarrow \mathbb{V}(M)$$ of superschemes induced by $S^\punto_A M\xrightarrow{s_0}  A\,.$
\end{definition}

We finish this section recalling the notion of relative dualizing complex and some remarkable properties. Let $i\colon X\hookrightarrow Y$ be a closed immersion of codimension $d$ between two (super) schemes.  The direct image
\[i_*\colon \Dqc(X)\to \Dqc(Y)\] has a right adjoint
\[ i^!\colon \Dqc(Y)\to \Dqc(X) \] and one has
\[ i_*i^!\M=\RR\HHom^\punto_{\OO_Y}(i_*\OO_X,\M) \] because the functor $i_*i^!$ is right adjoint of the functor $i_*\LL i^*$ and this one is tensoring by $i_*\OO_X$ by projection formula.
If $j\colon T\hookrightarrow X$ is another closed immersion, then $i^!\circ j^!=(j\circ i)^!$.

The complex $i^!\OO_Y$ is denoted by $D_{X/Y}$ and named the {relative dualizing complex} of $X$ over $Y$. One has
\[ i_* H^i(D_{ X/Y})=\HExt^i_{\OO_Y}(i_*\OO_X,\OO_Y).\] In particular, the sheaf $H^d(D_{X/Y})$ is called the {relative dualizing sheaf} of $X$ over $Y$ and denoted by $\omega_{X/Y}$. One has: $i_*\omega_{X/Y}=\HExt^d_{\OO_Y}(i_*\OO_X,\OO_Y)$.

\begin{prop} [Transitivity]\label{transitivity} For any   complex $\E$ on $Y$ one has a natural  morphism
\[ \LL i^*\E\overset\LL\otimes_{\OO_X}D_{X/Y}\to i^!\E \] which is an isomorphism if $\E$ is perfect.
In particular, if $j\colon Y\hookrightarrow Z$ is a closed immersion and $D_{Y/Z}$ is perfect, one obtains an isomorphism
\[  \LL i^*D_{Y/Z}\overset\LL\otimes_{\OO_X} D_{X/Y}  =D_{X/Z}.\]
\end{prop}

\begin{proof} The projection formula $i_*(\LL i^*\E\overset\LL\otimes_{\OO_X}D_{X/Y})= \E\overset\LL\otimes_{\OO_Y}i_*D_{X/Y}$ and the unit morphism $i_*D_{X/Y}\to \OO_Y$ define a morphism $i_*(\LL i^*\E\overset\LL\otimes_{\OO_X}D_{X/Y})\to\E$, hence, by adjunction, a morphism $\LL i^*\E\overset\LL\otimes_{\OO_X}D_{X/Y}\to i^!\E$. After applying $i_*$ (and projection formula) one obtains the natural morphism $$\E\overset\LL\otimes\RR\HHom_{\OO_Y}(i_*\OO_X,\OO_Y)\to \RR\HHom_{\OO_Y}(i_*\OO_X,\E)$$ which is an isomorphism if $\E$ is perfect.

If $j\colon Y\hookrightarrow Z$ is a closed immersion and $D_{Y/Z}$ is perfect, then
\[ D_{X/Z}=(j\circ i)^!\OO_Z= i^!(j^!\OO_Z)=i^!D_{Y/Z}= \LL i^*D_{Y/Z}\overset\LL\otimes_{\OO_X} D_{X/Y}.\]
\end{proof}

\begin{prop}[Flat base change]\label{flatbasechange}
   Let us consider a cartesian square
\[ \xymatrix{ \overline X\ar[r]^{\bar i}\ar[d]_{f_X} & \overline Y\ar[d]^f \\ X\ar[r]^i & Y}\] with $f$ flat. Then, one has a natural isomorphism
of functors
\[i^!_f\colon  f_X^*i^!\M\to {\bar i}^!f^*\M.\] In particular, one has an isomorphism
\[ D_{X/Y}(f)\colon f_X^*D_{X/Y}\to D_{\overline X /\overline Y }\] and then a sheaf isomorphism
\[\omega_{X/Y}(f)\colon f_X^*\omega_{X/Y}\to \omega_{\overline X/\overline Y}.\]
\end{prop}
\begin{proof} (sketch): The base change isomorphism $f^*i_*={\bar i}_*f_X^*$ induces an isomorphism $f^*i_*i^!={\bar i}_*f_X^*i^!$. On the other hand, the unit morphism $i_*i^!\to \id$, induces a morphism $f^*i_*i^!\to  f^*$. Thus, one obtains a morphism ${\bar i}_*f_X^*i^!\to f^*$ and, by adjunction, a morphism $f_X^*i^! \to {\bar i}^!f^*$. Applying ${\bar i}_*$, one obtains the natural morphism \[ f^*\RR\HHom_{\OO_Y}(i_*\OO_X,\underline\quad)\to \RR\HHom_{\OO_{\overline Y}}({\bar i}_*\OO_{\overline X},f^*(\underline\quad))\] which is an isomorphism, because   if $\E \to i_*\OO_X$ is a resolution of $i_*\OO_X$ by locally free finite $\OO_Y$-modules, then $f^*\E $ is a resolution of ${\bar i}_*\OO_{\overline X}$ by locally free finite $\OO_{\overline Y}$-modules and one has an isomorphism of complexes 
\[ f^*\HHom^\punto_{\OO_Y}(\E ,\M)\overset\sim\to \HHom^\punto_{\OO_{\overline Y}}(f^*\E ,f^*\M).\]
\end{proof}

\begin{remark} Applying ${\bar i}_*$ (and base change) to $i^!_f\colon  f_X^*i^!\M\to {\bar i}^!f^*\M$ one obtains the natural isomorphism
\[ f^*\RR\HHom_{\OO_Y}(i_*\OO_X,\M)\overset\sim\to \RR\HHom_{\OO_{\overline Y}}({\bar i}_*\OO_{\overline X},f^*\M).\]
\end{remark}

We leave the reader to check the following:

\begin{remarks}\label{remarks-transitivity} The flat base change isomorphism $i^!_f$ is compatible with transitivity and functorial on $f$ and $i$: That is:

 (1) By Proposition \ref{transitivity}, for any   complex $\E$ on $Y$ one has morphisms (let us denote $\E_{\overline X}=f_X^*\LL i^*\E = \LL {\bar i}^* f^*\E$) 
\[\aligned \E_{\overline X}\overset\LL \otimes f_X^*D_{Y/X} &= f_X^*(\LL i^*\E\overset\LL\otimes D_{Y/X})\overset{f_X^*(\ref{transitivity})}\longrightarrow  f_X^*i^!\E   \\
\E_{\overline X}\overset\LL \otimes D_{\overline X/\overline Y} &= \LL {\bar i}^* f^*\E\overset\LL\otimes D_{\overline X/\overline Y} \overset{\ref{transitivity}}\longrightarrow  {\bar i}^! f^*\E.\endaligned\] These  morphisms are compatible with the ``vertical'' morphisms $\id\otimes D(f)$ and $i^!_f$ (i.e., one has a commutative diagram).

(2) Let $T\overset j\hookrightarrow X$ be another closed immersion, $\overline T =T\times_X \overline X$ and $f_T\colon \overline T \to T$. One has a commutative diagram
\[\xymatrix{ f_T^*j^!i^!\M\ar[rr]^{(i\circ j)^!_f}\ar[dr]_{j^!_{f_X}} & & {\bar j}^!{\bar i}^!f^*\E \\ & {\bar j}^!f_X^* i^!\M \ar[ur]_{{\bar j}^!(i^!_f)} &
}
\]

(3) Let us consider a second cartesian square \[ \xymatrix{ \widetilde X \ar[r]^{\widetilde i}\ar[d]_{g_{\overline X}} & \widetilde Y \ar[d]^g \\ \overline X \ar[r]^{\bar i} & \overline Y }\] with $g$ flat.   Then one has a commutative diagram 
\[\xymatrix{ g_{\overline X}^* f_X^* i^!\M =(f_X\circ g_{\overline X})^*i^!\M\ar[dr]_{g_{\overline X}^*(i^!_f)} \ar[rr]^{i^!_{g\circ f}} & & {\widetilde i}^!(f\circ g)^*\M= {\widetilde i}^!g^*f^*\M \\
 & g_{\overline X}^* {\bar i}^!f^*\M \ar[ur]_{{\bar i}^!_g} & 
}
\] i.e., $ i^!_{g\circ f}= {\bar i}^!_g\circ g_{\overline X}^*(i^!_f)$.
\end{remarks}

\section{De Rham and Koszul complexes of a supermodule}\label{sec:DeRahmKos}
Let $M$ be a $A$-supermodule. For the sake of simplicity we shall denote $B=S^\punto_A M$ and by $B_n=S^n M$ its homogeneous component of degree $n$. The tensor product $B\otimes_{A} B$ is a graded $A$-supercommutative algebra, where the $\bbZ$-grading is $$(B\otimes_{A}B)_n=\underset{p+q=n}\oplus B_p\otimes_{A}B_q\,.$$ Moreover, the diagonal morphism $B\otimes_{A}B\to B$ is a (degree $0$) homogeneous morphism of graded superalgebras, so that the kernel $\Delta$ is a homogeneous ideal and $\Delta/\Delta^2=\Omega_{B/A}$ is a graded $B$-supermodule. If $b_p$ and $b_q$ are homogeneous elements of $B$ of degree $p$ and $q$ respectively, then $(\di b_q)b_p$ is an element of $\Omega_{B/A}$ of degree $p+q$. Following the above notation, $\Omega^p_{B/A}=\bigwedge^p_{B} \Omega_{B/A}$ denotes the $p$-th exterior power of $\Omega_{B/A}$ and is a graded $B$-supermodule in a natural way. We also denote by $B[-1]$ the shift with respect to the $\bbZ$-grading structure and by $[\Omega_{B/A}]_n$ the $\bbZ$-homogeneous part of $\Omega_{B/A}$ of degree $n$.

\begin{thm}\label{thm:relativediff}
Let $M$ be an $A$-supermodule and $B=S^\punto M$. The natural morphism of graded $B$-supermodules
$$\aligned M\otimes_A B[-1]&\to \Omega_{B/A}\\ m\otimes b&\mapsto (\di m)b\endaligned $$
is an isomorphism. Hence $\Omega^p_{B/A}\simeq \bigwedge^p_A M\otimes_A B[-p]$.
\end{thm}
\begin{proof}
By Equation \ref{eq:Der} and base change, it is enough to show that,  for any $B$-supermodule $N$,  $\HHom_A(M,N)$ is isomorphic as $B$-supermodule to $\SDer_A(B,N)$. Any  morphism of $A$-supermodules $f\colon M\to N$ induces a superderivation $S^\punto_A M\to N$, defined by the Leibniz rule. Conversely, if $D\colon S^\punto_A  M\to N$ is a superderivation, then the restriction  to $M$ defines a morphism $M\to N$ of $A$-supermodules. Both  assignments are inverse of each other. In particular, for $N=\Omega_{B/A}$, the canonical derivation $\di\colon B\to \Omega_{B/A}$ gives rise to the isomorphism $M\otimes _A B\to \Omega_{B/A}$ of $B$-supermodules, which is degree zero homogeneous after shifting B by $-1$.
\end{proof}

Let $D\colon B\to B$ be the superderivation consisting in multiplication by $n$ on $B_n=S^n_A M$. For each $p$, the interior product  $i_D\colon \Omega^{p}_{B/A}\to \Omega^{p-1}_{B/A}$ is a homogeneous of degree zero and parity preserving morphism of graded $B$-supermodules. On the other hand, the exterior differential $\di\colon \Omega^p_{B/A}\to \Omega^{p+1}_{B/A}$ is a homogeneous of degree zero and parity preserving morphism of graded $A$-supermodules, but not $B$-linear. Then one has the super Koszul and the super De Rham complexes:

\begin{definition}\label{def:KoszulDeRham}
Let $A$ be a supercommutative ring and $M$ an $A$-supermodule.
\begin{enumerate}
\item  The super Koszul complex associated to $M$, denoted by $\SKos_A (M)$, is the complex of graded $B$-supermodules $(\Omega^{\bullet}_{B/A}, i_D)$:
\begin{equation}\label{eq:relativeKoszul}
\xymatrix@C=17pt{
\cdots \ar[r] & \Omega^{p}_{B/A}\ar[r]^(.5){i_D}&\Omega^{p-1}_{B/A}\ar[r]^(.6){i_D}&\cdots\ar[r]^(.4){i_D}&\Omega_{B/A}\ar[r]^(.55){i_D}& B\ar[r]  & 0
}
\end{equation}

\item The super  De Rham complex of $M$ is the complex of graded $A$-supermodules $(\Omega^{\bullet}_{B/A}, \di)$:

 $$\SDeRham_A(M)\equiv \xymatrix@C=17pt{
0 \ar[r] &B\ar[r]^(.35)\di & \Omega_{B/A}\ar[r]^(.55){\di}&\cdots\ar[r]^(.4){\di}&\Omega^{p}_{B/A}\ar[r]^(.45){\di}&\Omega^{p+1}_{B/A}\ar[r]^(.55){\di}&\cdots
}$$
\end{enumerate}
There exist natural morphisms $A\to \SDeRham_A (M)$ and $\SKos_A (M)\to A$ induced, respectively, by the natural inclusion $A\to B$ and the cokernel of $i_D\colon \Omega_{B/A}\to  B$, that is, the projection $S^\punto_A M\to A$ on the zero degree component. 
\end{definition}

Taking the homogeneous components of degree $n\geq 0$ and using the Theorem \ref{thm:relativediff} we obtain the complexes of $A$-supermodules

\begin{equation*}
\SKos_A (M)_n\equiv \xymatrix@C=17pt{
0\ar[r]& \bigwedge^{n}_A M\ar[r]&\cdots\ar[r]&  \bigwedge^{p}_A M\otimes_{A} S_{A}^{n-p}M\ar[r]&\cdots\ar[r]  &S_{A}^n M\ar[r] & 0
}
\end{equation*}
and 
$$\SDeRham_A(M)_n\equiv \xymatrix@C=17pt{
0 \ar[r] &S^n_A M\ar[r]&\cdots\ar[r]&   \bigwedge^{p}_A M\otimes_{A} S_{A}^{n-p}M\ar[r]&\cdots \ar[r] &\bigwedge^{n}_A M\ar[r] & 0
}$$
and isomorphisms of complexes of $A$-supermodules 
$$\SKos_A (M)\simeq \underset{n\geq 0}\oplus\SKos_A (M)_n\quad \SDeRham_A (M)\simeq \underset{n\geq 0}\oplus\SDeRham_A (M)_n\,.$$

\begin{remark}\label{rem:Koszul-DeRham}
If $L=A^p$ is a free module of rank $(p\vert 0)$, then, taking into account Equation \ref{eq:simetricotwisted}, one obtains
$$
\aligned \SKos_A(\Pi L)_n &\simeq \Pi^n \SDeRham_A(L)_n [n]
\\ \SDeRham_A(\Pi L)_n &\simeq \Pi^n \SKos_A(L)_n [-n].
\endaligned
$$
\end{remark}

By Propositions \ref{prop:SEproperties} and \ref{prop:Diffproperties} the following properties hold:
\begin{prop}\label{prop:KosProperties}
\hspace{2em}
\begin{enumerate}
\item Base change: For any morphism of superrings $A\to A'$ there exist natural isomorphisms 
$$
\aligned \SKos_{A'} (M\otimes_A A')
&\simeq  \SKos_A (M)\otimes _A A'\\
\SDeRham_{A'} (M\otimes_A A')&\simeq \SDeRham_A (M)\otimes_A A'
\endaligned$$
\item Additivity: Let $M$ and $N$ be two $A$-supermodules and denote $B=S^\punto_A M$ and $B'=S^\punto_A N$. There exist natural isomorphisms 
$$
\SKos_A (M\oplus N)\simeq \SKos_{A} (M)\otimes_A \SKos_{A} (N)$$
(as complexes of $B\otimes_A B'$-supermodules, with $B=S^\punto_A M$  and  $B'=S^\punto_A N$) and 
$$
\SDeRham_A(M_1\oplus M_2)\simeq \SDeRham_A(M_1)\otimes_A \SDeRham_A(M_2)
$$
\end{enumerate}
\end{prop}
As a consequence, when $L$ is a free $A$-supermodule, say $L\simeq A^p\oplus\Pi A^q$, we have
$$
\aligned
\SKos_A(L)&\simeq\SKos_A(A)^{\otimes_A^p}\otimes_A \SKos_A(\pi A)^{\otimes_A^q}\\
\SDeRham_A(L)&\simeq\SDeRham_A(A)^{\otimes_A^p}\otimes_A  \SDeRham_A(\pi A)^{\otimes_A^q}
\endaligned
$$
Therefore, by Remark \ref{rem:Koszul-DeRham}, the explicit description of $\SKos_A(L)$ and $\SDeRham_A (L)$ boils down to considering the following two cases:
\begin{enumerate}
\item
$L\simeq A$, then $B=A[x]$ and one has the exact sequence
\begin{equation}\label{eq:KosAp0}
\xymatrix@R=3pt{
0\ar[r] & A[x]dx\ar[r]^{\cdot\,x} &  A[x]\ar[r] & A\ar[r] &0\\
&p(x)dx\ar@{|->}[r] & xp(x) &&
}
\end{equation}

\item 
$L\simeq \Pi A$, then $B=A[\theta]$, with $\theta$ of degree $1$ and $\vert\theta\vert=1$, hence $\theta^2=0$. Now, $\Omega_{B/A}=B d\theta$ (with $\vert d\theta\vert=1$) and $$\Omega^p_{B/A} = B\cdot (d\theta\wedge\overset p\cdots\wedge d\theta)$$ with $\vert d\theta\wedge\overset p\cdots\wedge d\theta\vert \equiv p \text{ mod } 2$. Let us denote $(d\theta)^p=d\theta\wedge\overset p\cdots\wedge d\theta$. One has $D=\theta\frac{\partial}{\partial\theta}$ and $i_D[ (\di\theta)^p]=p \theta(d\theta)^{p-1}$, so that the super Koszul complex $\SKos_A(\Pi A)$ reads: 
\begin{equation}\label{eq:KosA0q}
\xymatrix{
\cdots\ar[r]& A[\theta](d\theta)^2\ar[r]^{\cdot\, 2\theta}&A[\theta](d\theta)\ar[r]^{\cdot\, \theta} &A[\theta]\ar[r] &0
}
\end{equation}
where $A[\theta](d\theta)^p\xrightarrow{\cdot\, p\theta}A[\theta](d\theta)^{p-1}$ maps $q(\theta)(d\theta)^p$ to $p\theta q(\theta)(d\theta)^{p-1}$.
\end{enumerate}

\begin{remark}
If $A$ is a $k$-algebra, with $k$ a superring without odd elements (that is, a ring in the usual sense) and $M\simeq A^n\oplus\Pi A^m$, then 
$\SKos_A (M)\simeq (\Kos_{k} k^n\otimes_{k} \Kos_{k} {\Pi} k^m)\otimes_{k} A$.
In particular, $\SKos_A (A^n\oplus\Pi A^m)\simeq (\Kos_{\bbZ} \bbZ^n\otimes_{\bbZ} \Kos_{\bbZ} {\Pi} \bbZ^m)\otimes_{\bbZ} A$.
\end{remark}

\subsection{Acyclicity of super Koszul complex and super De Rham complex}\label{sec:aciclicity}
Next aim is to study when $\SKos_A (M)$ and $\SDeRham_A (M)$ are resolutions of $A$, that is, when $\SKos_A (M)\to A$  and $A\to\SDeRham_A (M)$ are quasi-isomorphisms. Note that $\Kos_A (M)_0=A=\SDeRham_A (M)_0$ and $\Kos_A (M)_n=\SDeRham_A (M)_n=0$, for all $n<0$.
\begin{thm}\label{thm:aciclicity}
Let $n$ be a positive integer number. If $n$ is invertible in $A$, then $\SKos_A (M)_n$ and $\SDeRham_A (M)_n$ are acyclic. Consequently, if $A$ is a $\bbQ$-algebra, then $\SKos_A (M)\to A$ is a left resolution of $A$ by  $B$-supermodules and $A\to \SDeRham_A (M)$ is a right resolution of $A$ by $A$-supermodules.
\end{thm}
\begin{proof}
By Cartan's formula (see Example \ref{ex:Lie}) $i_D\circ d+d\circ i_D= n Id$ on $[\Omega_{B/A}]_n$, then both $(\Kos_A M)_n$ and $(\SDeRham_A M)_n$ are homotopically trivial and therefore acyclic.
\end{proof}

For (some) free $A$-supermodules one can explicitly compute  the cohomology of its super Koszul complex
\begin{thm}\label{thm:Koszulfreemodule}
Let $L=A^p\oplus \Pi A^q$ a free $A$-supermodule of rank $(p|q)$. One has:
\begin{enumerate}
\item If $q=0$, then $\SKos_A (L)$ is a resolution of $A$.
\item If $p=0$, then $\SDeRham_A (L)$ is a resolution of $A$.
\item If $p=0,\, q=1$, then 
$$H^{-i}\big(\SKos_A (L)\big)=\Pi^iA_{i-tor}\oplus\Pi^{i+1}A/(i+1)A$$
where $A_{i-tor}$ denotes the $i$-torsion of $A$, that is, the elements $a\in A$ such that $ia=0$.
\end{enumerate}
\end{thm}
\begin{proof}
\begin{enumerate}
\item Since $\SKos_A (A^p)\simeq \SKos_A (A)\otimes_A\SKos_A (A^{p-1})$, we proceed by induction on $p$. For $p=1$, Equation \ref{eq:KosAp0} yields that $\SKos_A(A)$ is a resolution of $A$ by free $A$-supermodules. Assume $p\geq 1$, then by induction hypothesis we have that $\SKos_A (A^{p-1})_n$ is an acyclic complex of free $A$-supermodules, for all $n>0$. Therefore $[\SKos_A A\otimes_A\SKos_A A^{p-1}]_n$ is also acyclic for all $n>0$ and the result holds.
\item Follows from case $(1)$ and Remark \ref{rem:Koszul-DeRham}.
\item The De Rham complex of $A$ is 
$$0\to A[x]\xrightarrow{\di} A[x]dx\to 0$$ 
whose cohomology is
$$H^0\big(\SDeRham_A (A)\big) =\bigoplus_{n\geq 0}A_{n-\text{tor}}\, x^n\,\quad\text{ and }\, \quad H^1\big(\SDeRham_A (A)\big)=\bigoplus_{n>0} (A/nA) \, x^{n-1}\di x\,. $$
Then, by using the Remark \ref{rem:Koszul-DeRham} we obtain that
$$
\aligned 
H^{-i}\big(\SKos_A(\Pi A)\big) &= \Pi^{i}  A_{i-\text{\rm tor}} \oplus \Pi^{i+1}\,  A/(i+1)A\\
&=A_{i-{\rm\text{tor}}}(\di\theta)^i\oplus (A/(i+1)A)\, \theta(\di\theta)^i
\endaligned
$$
\end{enumerate}
\end{proof}
\begin{ex} $H^0(\SKos_\bbZ\Pi \bbZ)= \bbZ$ and $H^{-i}(\SKos_\bbZ \Pi \bbZ)= \Pi^{i+1}\, \bbZ/(i+1)\bbZ$ for $i>0$.
\end{ex}

\subsection{Universality of Koszul complex}
Before proceeding with further general facts to be obtained from the super Koszul complex, let us recall the classical definition of Koszul complex and interpret the complex defined in Definition \ref{def:KoszulDeRham} as a universal object.

Let $M$ be an $A$-module and let $\omega\colon M\to A$ a morphism. The Koszul complex associated to $\omega$ is the complex 
\begin{equation}\label{eq:KosClassical}
\Kos(M,\omega)\equiv \xymatrix@C=17pt{
\cdots \ar[r] & \bigwedge^{p}_{A} M\ar[r]^(.5){d_\omega}&\bigwedge_A^{p-1}M\ar[r]^(.6){d_\omega}&\cdots\ar[r]^(.3){d_\omega}&\bigwedge_{A}^2 M\ar[r]^(.55){d_\omega}&M \ar[r]^\omega &A\ar[r]& 0
}
\end{equation}
where the differential $d_\omega\colon  \bigwedge^{\punto}_{A} M\to  \bigwedge^{\punto}_{A} M$ is the unique anti-derivation of degree $-1$ such that $d_\omega(m)=\omega(m)$ for all $m\in M$. Explicitly: 
$$d_\omega(m_1\wedge\cdots\wedge m_p)=\displaystyle\sum_{i=1}^{p}(-1)^{i+1}\omega(m_i)\,m_1\wedge\cdots\wedge\widehat{m_i}\wedge\cdots\wedge m_p\,.$$ 
As a particular case, if $M\simeq A^r$ is a free module and $\omega\colon A^r\to A$ is given by a sequence $(a_1,\dots,a_r)$, then the Koszul complex associated to $\omega$ is called the Koszul complex of the sequence $\textit{\textbf{a}}=(a_1,\dots,a_r)$ and is denoted by $\Kos(\textit{\textbf{a}})$. Moreover, if $\textit{\textbf{a}}$ is a regular sequence, then $\Kos(\textit{\textbf{a}})$ is a resolution of $A/(a_1,\dots,a_r)$. Therefore the Koszul complex introduced in the Definition \ref{def:KoszulDeRham} is, in this notation, the Koszul complex $\Kos(\Omega_{B/A},i_D)$. If $M\simeq A^r$ is a free $A$-module, then $S^\punto_A M\simeq A[x_1,\dots, x_r]$ and $\Kos(\Omega_{B/A},i_D)$ is the Koszul complex of the sequence $(x_1,\dots,x_r)$.

Consider now the Koszul complex $\Kos(M\otimes_A B, \theta)$ associated to the natural morphism $\theta\colon M\otimes_A S^\punto M\to S^\punto M$, that is, denoting $B=S^\punto_A M$
\begin{equation}\label{eq:KosClassicalB}
\Kos(M\otimes_A B, \theta)\equiv \xymatrix@C=17pt{
\cdots \ar[r] &\bigwedge_B^{p}(M\otimes_A B)\ar[r]^(.7){d_\theta}&\cdots\ar[r]^(.3){d_\theta}&\bigwedge_{B}^2 (M\otimes_A B)\ar[r]^(.55){d_\theta}&M\otimes_A B \ar[r]^(.65)\theta &B\ar[r]& 0
}
\end{equation}
This Koszul complex is universal in the sense that any morphism of modules $\omega\colon M\to A$ yields a morphism $s_\omega\colon B\to A$ of rings in such a way that 
$$\Kos(M\otimes_A B,\theta)\otimes_B A\simeq \Kos(M,\omega)\,.$$
Thus, taking into account the Theorem \ref{thm:relativediff} one finds that 
$$\Kos_A (M)\otimes_B A\simeq \Kos(M,\omega)\,. $$
Note that, even though $\Kos_A(M)$ is acyclic, the Koszul complex $\Kos(M,\omega)$ is either acyclic or not depending on how the base change $B\to A$ is. From a geometric perspective this can be understood as follows: let $\sigma_0\colon\SSpec A\hookrightarrow\mathbb{V}(M)$ be the zero section of $\pi\colon\mathbb{V}(M)\to \SSpec A$ induced by $s_0\colon S^\bullet_A M\to A$ and let $\sigma_\omega\colon \SSpec A\to \mathbb{V}(M)$ be the section of $\pi$ induced by $s_\omega\colon S^\punto_A M\to A$. For clearness let us put $A_0$ (resp. $A_\omega$) to refer to $A$ with the $B$-module structure given by $s_0$ (resp. $s_\omega$). Then 
$$\Tor^B_i(A_0,A_\omega)=H^{-i}(\Kos_A (M)\otimes_B A)=H^{-i}(\Kos(M,\omega))$$
so that the acyclicity of $\Kos(M,\omega)$ depends on how $\sigma_0$ and $\sigma_\omega$ intersect each other.

As applications of the super Koszul complex we shall address in the next sections the notions of Berezinian of a free supermodule and the Berezin function of an automorphism of a free supermodule as well as the problem of extending the so-called Bott formula to the projective superspace.

\section{The Berezinians}\label{sec:Berezinians}
The Berezinian (viewed as a complex or as a module) plays the role of the dualizing complex (respectively dualizing sheaf) of an ordinary space, which makes no sense in the super setting since $\bigwedge^i\Omega_X$ does not need vanish from certain $i$. In this section we revisit the notions of the Berezinian of a free supermodule and the Berezin function by using the super Koszul complex and the relative dualizing sheaf of a closed immersion.  It is worth noting, though, that it is not the first time this issue is discussed (for previous approaches we refer to \cite{OP84} and \cite{NR22}). Even so, let us stress that here this problem is addressed in an intrinsic way (we avoid to use coordinates as much as possible) and in a fairly general point of view (we do not need assumptions neither on smoothness nor on the characteristic of the base superring).

\begin{definition}
Let $M$ be an $A$-supermodule. The dual module $M^*$ of $M$ is defined as 
$$M^*={\HHom}_A(M,A)\,.$$
which has
a natural structure of $A$-supermodule (by composition) taking into account that
$A_0=\Hom_A(A,A)=\Hom_A(\Pi A,\Pi A)$ and $A_1=\Hom_A(A,\Pi A)=\Hom_A(\Pi A,A)$.
\end{definition}
Notice that $A^*=A$. The universal property of the dual module is given by the formula
$$\HHom_A(N,M^*)=\HHom_A(N\otimes_A M, A)\,.$$
Any morphism of $A$-supermodules $M\xrightarrow{f}{N}$ induces a morphism of $A$-supermodules $N^*\xrightarrow{f^*}{M^*}$ which is called the {\em transpose} of $f$.
\begin{prop}
Some elementary properties of the dual are:  \begin{enumerate}
\item $(M\oplus N)^*=M^*\oplus N^*$.
\item $(\Pi M)^*=\Pi(M^*)$.
\item If $L$ is a free supermodule of rank $(p\vert q)$, so is $L^*$.
\item One has  a natural morphism $M^*\otimes_AN^*\to (M\otimes_AN)^*$ which is an isomorphism if $M$ and $N$ are free supermodules of finite rank.
\item One has natural morphisms
\[ S^n_A(M^*)\to (S^n_AM)^*,   \qquad {\bigwedge}^n_A(M^*)\to ({\bigwedge}^n_AM)^* \] which are isomorphisms if $M$ is free and finite.
\item For any $A$-supermodule $M$ and any superring homomorphism $A\to B$, there is a natural morphism
\[{\HHom}_A(M,A)\otimes_AB\to \HHom_B(M\otimes_AB,B)\] which is an isomorphism if $M$ is free and finite.
\end{enumerate}
\end{prop}

\begin{definition}\label{def:Berezinian cplx}
Let $M$ be an $A$-supermodule and $B=S_A^\bullet M$. The Berezinian complex of $M$ is defined to be 
$$
\BBer^\bullet_A(M)=\HHom_B(\SKos_A (M),B)=\HHom^\bullet_B(\SKos_A (M),B)\,.
$$
\end{definition}

By definition, $\BBer^\bullet_A(M)$ is a complex of graded $B$-supermodules whose $p$-th component is $$\BBer^p_A(M)=\HHom_B(\Omega^p_{B/A},B)=\HHom_B({\exto}^p_AM\otimes_A B[-p],B)\,.$$
\begin{ex}
If $L$ is a free $A$-supermodule of finite rank, then $\hSKos{-p}_A L={\bigwedge}_A^pL\otimes_A B[-p]$ and $$\HHom_B({\bigwedge}^p_AL\otimes_A B[-p],B)={\bigwedge}_A^pL^*\otimes_A B[p]\,.$$ So that
$$\BBer^\bullet_A(L)\equiv  \xymatrix@C=17pt{
0\ar[r]&B\ar[r]&L^*\otimes_A B[1]\ar[r]&\bigwedge_{A}^2 L^*\otimes_AB[2]\ar[r]&\cdots}
$$
and its homogeneous component of degree n is
$$[\BBer^\bullet_A(L)]_n\equiv  \xymatrix@C=17pt{
0\ar[r]&S^n_AL\ar[r]&L^*\otimes_A S^{n+1}_AL\ar[r]&\bigwedge_{A}^2 L^*\otimes_A S^{n+2}_AL\ar[r]&\cdots}
$$
\end{ex}

As straightforward consequences of  Proposition \ref{prop:KosProperties} we have
\begin{prop}\label{prop:Dual}
\hspace{2em}
\begin{enumerate}
\item Base change: Let $L$ be a free $A$-supermodule of finite rank. Let $A\to A'$ be a morphism of superrings and put $L'=L\otimes_A A'$ and $B'=S^\punto_AL'$. Then 
$$\BBer^\bullet_{A'}(L')\simeq \BBer^\bullet_A(L)\otimes_B B'\simeq \BBer^\bullet_{A}(L)\otimes_A A'\,.$$
\item Additivity: Let $L_1$ and $L_2$ be two free $A$-supermodules of finite rank. Then $$\BBer^\bullet_A(L_1\oplus L_2)\simeq\BBer^\bullet_A(L_1)\otimes_A\BBer^\bullet_A(L_2)$$ (as complexes of $B_1\otimes_A B_2$-supermodules, with $B_i=S^\punto_A L_i$).
\end{enumerate}
\end{prop}

The next step is to show that the Berezinian of a free $A$-supermodule is a concentrated complex, that is, $\BBer^\bullet_A(L)$ is just a graded supermodule, up to shift. 

\begin{thm}\label{thm:CohomBerFree}
Let $L\simeq A^p\oplus\Pi A^q$ be a free $A$-supermodule of rank $(p|q)$. Then 
\begin{enumerate}
\item $$H^i\big(\BBer^\bullet_A(L)\big)= {\bigwedge}^p_A (A^p)^* \otimes_A H^{i-p}\big(\BBer^\bullet_A(\Pi A^q)\big)$$
and 
$$H^p \big(\BBer^\bullet_A(L)\big) = {\bigwedge}^p_A (A^p)^* \otimes_A S^q_A(\Pi A^q) =  {\bigwedge}^p_A (A^p)^* \otimes_A \Pi^q{\bigwedge}^q_A (A^q)=\begin{cases} A& \text{ if }q \text{ is even}\\ \\    \Pi A &\text{ if }q \text{ is odd}\end{cases}
$$
\item If $A$ is a $\bbQ$-algebra, then
$ H^i\big(\BBer^\bullet_A(L)\big) = 0$ for any $i\neq p$.
\end{enumerate}
\end{thm}

\begin{proof}
\begin{enumerate}
\item By Proposition \ref{prop:Dual}, it sufficies to show that 
$$H^i(\BBer^\bullet_A(A))=\left\{ \aligned 0,\text{ if }i\neq 1\\ \\ A,\text{ if }i=1\endaligned\right.  \quad\text{and}\quad H^0(\BBer^\bullet_A(\Pi A)) =\Pi A\,.$$
In the first case, $B=A[x]$, with $x$ even, and the Koszul complex is
$$0\to A[x]\di x\to A[x]\to 0, \quad \di x\mapsto x$$ hence $\BBer^\bullet_A(A)$ is the complex
$$0\to A[x]\to A[x]\frac\partial{\partial x}\to 0, \quad 1\mapsto x\, \frac\partial{\partial x}$$ and the result follows.

In the second case, $B=A[\theta]$, with $\theta$ odd, and the Koszul complex is
$$\cdots \to B\overset{ 2\theta\cdot}\to B\overset{\theta\cdot}\to B\to 0$$ hence $\BBer^\bullet_A(\Pi A)$ is the complex
$$0\to B\overset{\theta\cdot}\to B\overset { 2\theta\cdot}\to B \to\cdots$$ and then $H^0(\BBer^\bullet_A(\Pi A))=\Ker[ B\overset{\theta\cdot}\to B]= A\theta$.

\item If $A$ is a $\bbQ$-algebra, then $H^i(\BBer^\bullet_A(\Pi A)=0$ for $i>0$, and we conclude by additivity.
\end{enumerate}
\end{proof}

\begin{definition}\label{def:Berezinian mod}
The Berezinian module of a free $A$-supermodule $L$ of rank $p|q$ is  
$$
\Ber_A(L):=H^p(\BBer^\bullet_A(L))={\bigwedge}^p_A (A^p)^* \otimes_A \Pi^q{\bigwedge}^q_A (A^q)
$$
which is a free $A$-supermodule of rank $(1|0)$ if $q$ is even and of rank $(0|1)$ if $q$ is odd. 
\end{definition}

From Proposition \ref{prop:Dual} one gets
\begin{prop}
Let $L$ and $L'$ be free $A$-supermodules of finite rank, and $A\to A'$ a morphism of superrings. Then: 
\begin{enumerate}
\item $\Ber_{A}(L)\otimes_A A'=\Ber_{A'}(L\otimes_A A')$.
\item $\Ber_{A}(L\oplus L')=\Ber_{A}(L)\otimes_A\Ber_{A}(L')$.
\end{enumerate}
\end{prop}

The Berezinian of a free supermodule can be understood as a dualizing sheaf in the following sense. Let $L$ be a free $A$-supermodule of rank $p\vert q$ and let $\bbL=\SSpec S^\punto_A L\xrightarrow{\pi}{\SSpec A=X}$ be the linear superscheme associated to $L$. The zero section $s \colon X\hookrightarrow \LL$ of $\pi$ is a closed immersion of codimension $p$. Thus, we have the relative dualizing complex  $D_{X/\LL}$ and the relative dualizing sheaf $ \omega_{X/\LL}= H^p(D_{X/\LL})=\HExt^p_{B}(A,B)$, being $B=S^\punto_A L$.

\begin{thm}\label{berziniano=dualizante} One has
\[\aligned H^i(D_{X/\LL})&=0, \text{ for }i\neq p \\ \omega_{X/\LL}&=\Ber_A(L).\endaligned\]
\end{thm}

\begin{proof} The natural morphism $\SKos_A(L)\to A$ induces a morphism $s_*D_{X/\LL}\to \BBer_A(L)$ and then a morphism $\omega_{X/\LL}\to \Ber_A(L)$. Let us proceed by induction on $p+q$. First, assume $p+q=1$. Then $L=A$ or $L=\Pi A$. 

If $L=A$, then $\SKos_A(L)\to A$ is a quasi-isomorphism, so $s_*D_{X/\LL}\to \BBer_A(L)$    is an isomorphism (in the derived category) and we are done.  

If $L=\Pi A$, then $B=A[\theta]$  and we have an exact sequence
\[ \cdots \to \Pi^nB\overset{\cdot\theta}\to \cdots \to \Pi^2B\overset{\cdot\theta}\to \Pi B \overset{\cdot\theta}\to B\to A\to 0\] and then
\[ \HHom_B(A,B)=\theta B= \Pi A,\quad \HExt^i_B(A,B)=0,\text{ for }i>0.\]

Now assume the theorem holds for $p+q<n$ and let $L$ be a free module with $p+q=n$. Put $L=L_1\oplus L_2$, with $L_i$ under the induction hypothesis. Let us denote $\LL_i=\Spec S^\punto_AL_i$, $\pi_i\colon\LL_i\to X$ the natural morphism and $s_i\colon X\hookrightarrow \LL_i$ the zero section. 

The morphism $L_1\to L $ induces a (flat) morphism $f\colon \LL\to \LL_1$ and the projection $L\to L_2$ induces a closed immersion $j\colon \LL_2\hookrightarrow \LL$;  one has a cartesian diagram
\[ \xymatrix{\LL_2\ar[r]^j\ar[d]_{\pi_2} & \LL\ar[d]^f\\ X\ar[r]^{s_1} &\LL_1}.\]  Then, by flat base change (Proposition \ref{flatbasechange}), $$ D_{\LL_2/\LL}=\pi_2^*D_{X/\LL_1}.$$ 
Now, the zero section $X\hookrightarrow \LL$ is the composition of  $s_2$ with $j$. Notice that $D_{\LL_2/\LL}=\pi_2^*D_{X/\LL_1}$ is perfect by induction then, by transitivity (Proposition \ref{transitivity}),  \[ D_{X/\LL}=\LL s_2^* D_{\LL_2/\LL}\overset\LL\otimes D_{X/\LL_2}=D_{X/\LL_1}\overset\LL \otimes D_{X/\LL_2}\] and we conclude by induction.
\end{proof}

The same argument given in the proof of Theorem \ref{berziniano=dualizante} proves the following
\begin{prop}  [Additivity]\label{additivity-2} If $0\to L_1\to L\to L_2\to 0$ is an exact sequence of finite free $A$-modules, then
\[ D_{X/\LL}=D_{X/\LL_1}\overset\LL\otimes_A D_{X/\LL_2} \] and then
\[ \omega_{X/\LL}=\omega_{X/\LL_1}\otimes_A \omega_{X/\LL_1}\qquad \text{ and }\qquad \Ber_A(L)=\Ber_A(L_1)\otimes_A\Ber_A(L_2).\]
\end{prop}

\subsection{The Berezin function}

A morphism of $A$-supermodule $\varphi\in\HHom(A^{p|q},A^{p|q})$ can be regarded (fixed homogeneous basis) as a matrix $X\in\Mat(p|q,A)$ of order $p+q$
$$
X=\left ( 
\begin{array}{c|c}
A & B \\
\hline 
C & D
\end{array}
\right )$$
Note that $\Mat(p|q,A)$ is a $A$-superalgebra by declaring that $X$ is even (resp. odd) if $A$ and $D$ have even (resp. odd) entries, while $B$ and $C$ have odd (resp. even) entries. Consider the subgroup of even invertible morphisms $\GL(p|q, A)\subseteq \Mat(p|q,A)_0$, that is an analogue to the general linear group in the ordinary sense. In particular, $\GL(1|0, A)=A_0^*$ (the group of invertible elements of $A_0$) and it is known that $X\in\GL(p|q, A)$ if and only if $A\in\GL(p,A_0)$ and $D\in\GL(q,A_0)$, that is, if and only if $A$ and $D$ are invertible as ordinary matrices with entries in $A_0$. Then, the supergeometric analog of the usual determinant on elements of $\GL(p|q, A)$ is the so called Berezin function (or just the Berezinian) and is defined to be the morphism
\begin{equation}\label{eq:Berezinfuntion}
\aligned
\Ber\colon \GL(p|q, A)&\to A_0^*\\
X&\mapsto \det(D^{-1})\det(A-BD^{-1}C)
\endaligned
\end{equation}

The Berezin function first appeared in \cite{Ber87} where Berezin considers an extension of the Jacobian in ordinary differential geometry to calculate the super version of the change of variable formula for integral forms on supermanifolds, likely the first topic of supergeometry in which the generalization from the purely even case is not obviuos. Let us show how to obtain the expression of the Berezin function. First, in analogy to how the determinant of an automorphism is defined in the purely even case, we shall give the definition of the Berezinian of an automorphism of a free supermodule by using  the Berezinian modules given above (Definition \ref{def:Berezinian mod}). For any even morphism $f\colon L_1\to L_2$ between two finite rank free $A$-supermodules let us denote by $f^*\colon L_2^*\to L_1^*$ the transpose morphism and by  $\bigwedge^p f\colon \bigwedge^p_A L_1\to  \bigwedge^p_A L_2$ the induced morphism between  the $p$-th exterior powers. We also denote by 
$S^\punto(f)\colon B_1\to B_2$ the induced morphism between their symmetric algebras, where $B_i= S^\bullet(L_i)$. One has a natural morphism $\SKos_A(L_1)\otimes _{B_1} B_2\to\SKos_A(L_2)$ which in degree $-p$ is 
$$
({\wedge}^p f)\otimes 1\colon {\bigwedge}^{p}_AL_1\otimes_A B_2[-p] \longrightarrow {\bigwedge}^{p}L_2\otimes_A B_2[-p]\,.
$$
Taking dual, one has the morphism 
\begin{equation}\label{1} 
\BBer_A^\bullet(L_2)\to \BBer_A^\bullet(L_1)\otimes_{B_1}B_2 
\end{equation} 
which, in degree $p$, is given by
\[ ({\wedge}^p f^*)\otimes 1\colon {\bigwedge}^{p}_AL_2^*\otimes_A B_2[p] \longrightarrow {\bigwedge}^{p}L_1^*\otimes_A B_2[p]\,.\]
On the other hand, one has the morphism
\begin{equation}\label{2} 
\BBer_A(L_1)=\BBer_A(L_1)\otimes_{B_1}B_1\overset{1\otimes S^\punto(f)}\to\BBer_A(L_1)\otimes_{B_1} B_2\,.
\end {equation}
If, in addition, $f\colon L_1\to L_2$ is an isomorphism, then the morphisms \ref{1} and \ref{2} so are and the composition of \ref{1} and the inverse of \ref{2} yields an isomorphism of complexes of $A$-supermodules
$$\BBer^\bullet(f):=({\bigwedge}^{\punto} f^*)\otimes S^\punto(f^{-1})\colon \BBer_{A}^\bullet(L_2)\to \BBer_{A}^\bullet(L_1)$$
which in degree $p$ is the isomorphism of 
$A$-supermodules 
$$\displaystyle{{\bigwedge}^{p} f^*\otimes S^\punto(f^{-1})\colon {\bigwedge}^p_A L_2^*\otimes_A B_2[p]\isom {\bigwedge}^p_A L_1^*\otimes_A B_1[p]}\,.$$

Since both symmetric and exterior algebra are functorial, the isomorphism $\BBer^\bullet(f)$ is compatible with composition $\BBer^\bullet(f\circ g)=\BBer^\bullet(g)\circ \BBer^\bullet(f)$ and identity $\BBer^\bullet(id_L)=id_{\BBer^\bullet_A(L)}$.

In particular, given a free $A$-supermodule $L$ of finite rank, any automorphism $\tau\colon L\isom L$ yields an automorphism, say $\BBer^\bullet(\tau)$, of  $\BBer_A^\bullet(L)$. Taking the $p$-th cohomology, one gets that the induced automorphism 
$$\Ber(\tau)\colon\Ber_A(L)\isom \Ber_A(L)$$ is the multiplication by a unit in $A_0$.

\begin{definition}
The Berezinian of an automorphism $\tau\colon L\to L$ of a free $A$-supermodule $L$ is defined to be the unit $\Ber(\tau)\in A_0$ corresponding to  $\Ber(\tau)\colon\Ber_A(L)\isom \Ber_A(L)$.
\end{definition}

\begin{prop}\label{berezinian-compute}
\hspace{2em}
\begin{enumerate}
\item The Berezinian is a morphism of groups: if $\tau,\tau'\in {\Aut}(L)$, then $\Ber(\tau\circ \tau')=\Ber(\tau)\cdot\Ber(\tau')$.
\item Additivity: Let $\tau_i\colon L_i\to L_i$ be an automorphism for any $i=1,2$, then $\Ber(\tau_1\oplus\tau_2)=\Ber(\tau_1)\cdot\Ber(\tau_2)$.
\item If $L$ is a free supermodule of finite and purely even rank, that is $L\simeq A^p$, then the Berezinian of $\tau\in \Aut_A(L)$ is the determinant $\det(\tau)$ of $\tau$.
\item If $L\simeq \Pi A^q$, then $\Ber(\tau)=\det(\tau)^{-1}$.
\item If $\tau\colon A^p\oplus\Pi A^q\to A^p\oplus\Pi A^q$ is given by $
\left ( 
\begin{array}{c|c}
A & 0 \\
\hline 
0 & D
\end{array}
\right )$, then $\Ber(\tau)=\det A\cdot \det D^{-1}$.
\end{enumerate}
\end{prop}	
\begin{proof}
\begin{enumerate}
\item It follows from functoriality of symmetric and exterior algebra.
\item  It is straightforward to check by using the additivity of symmetric algebra and exterior algebra.
\item If $L=A^p$, then $\bigwedge^p\tau^*\colon \bigwedge^p L^*\to \bigwedge^p L^*$ is  the determinant of $\tau$.
\item It follows taking into account that $S^q(\tau)=\bigwedge^p(\tau)^{-1}$.
\item It is consequence of $(3)$ and $(4)$.
\end{enumerate}
\end{proof}

The next step is to understand the Berezinian of an automorphism of a free supermodule as an isomorphism between dualizing sheaves. From that, the formula for the Berezin function (\ref{eq:Berezinfuntion}) arises in a natural way.

\begin{definition} Let $f\colon L_1\to L_2$ be an isomorphism of free modules and let us still denote by $f$ the induced isomophism $ \LL_2\to\LL_1$ between the spectra of their symmetric algebras. One has a cartesian diagram
\[ \xymatrix{X \ar[r]^{s_2}\ar[d]_{\id_X} & \LL_2\ar[d]^f \\ X\ar[r]^{s_1} & \LL_1}\] where $s_i\colon X\to\LL_i$ is the zero section. This defines by flat base change (Proposition \ref{flatbasechange}) an isomorphism
\[ \omega_{X/\LL_1}(f)\colon \omega_{X/\LL_1}=\id^*_X\omega_{X/\LL_1}\to\omega_{X/\LL_2}.\] In particular, if $\tau\colon L\to L$ is an automorphism, then $\omega_{X/\LL}(\tau)\colon \omega_{X/\LL}\to \omega_{X/\LL}$ is the multiplication by a unit of $A_0$, which is also denoted by $\omega_{X/\LL}(f)$.
\end{definition}

\begin{lem}\label{berziniano=dualizante2} Via the isomorphisms $\Ber_A(L_i)=\omega_{X/\LL_i}$, one has
\[ \omega_{X/\LL_1}(f)=\Ber(f^{-1}).\]
\end{lem}

\begin{proof} First notice that $\omega_{X/\LL_1}(f)$ coincides with its adjoint $\omega_{X/\LL_1}\to {\id_X}_*\omega_{X/\LL_2}$.

 Applying ${s_2}_*$ to $\omega_{X/\LL_1}(f)\colon \id^*_X\omega_{X/\LL_1}\to\omega_{X/\LL_2}$ one obtains the natural isomorphism $$\phi\colon f^*\RR\HHom_{\OO_{\LL_1}}({s_1}_*\OO_X,\OO_{\LL_1})\to \RR\HHom_{\OO_{\LL_2}}({s_2}_*\OO_X,\OO_{\LL_1})$$  whose adjoint $ \RR\HHom_{\OO_{\LL_1}}({s_1}_*\OO_X,\OO_{\LL_1})\to f_*\RR\HHom_{\OO_{\LL_2}}({s_2}_*\OO_X,\OO_{\LL_1})$ is the composition \[ \RR\HHom_{\OO_{\LL_1}}({s_1}_*\OO_X,\OO_{\LL_1})\to f_*f^*\RR\HHom_{\OO_{\LL_1}}({s_1}_*\OO_X,\OO_{\LL_1})\overset{f_*(\phi)}\to f_*\RR\HHom_{\OO_{\LL_2}}({s_2}_*\OO_X,\OO_{\LL_2})\] i.e., the composition
\[ \RR \HHom_{B_1}(A,B_1)\to \RR \HHom_{B_1}(A,B_1)\otimes_{B_1}B_2 \to \RR \HHom_{B_2}(A,B_2)\]  One has a commutative diagram
\[ \xymatrix{\RR \HHom_{B_1}(A,B_1)\ar[r]\ar[d] &\RR \HHom_{B_1}(A,B_1)\otimes_{B_1}B_2 \ar[r]\ar[d] & \RR \HHom_{B_2}(A,B_2) \ar[d] \\ \BBer_A(L_1)\ar[r] & \BBer_A(L_1) \otimes_{B_1}B_2 \ar[r] & \BBer_A(L_2)}\] and, by definition, the composition $\BBer_A(L_1)\to \BBer_A(L_1) \otimes_{B_1}B_2 \to \BBer_A(L_2)$ is the inverse of $\BBer_A(f)$.
\end{proof}

\begin{lem}\label{lema-berziniano} Let $0\to L_1\to L\to L_2\to 0$ be an exact sequence of free modules and let $\tau\colon L\to L$ be an automorphism leaving $L_1$ stable (i.e., $\tau(L_1)=L_1$). Let us denote $\tau_1=\tau_{\vert L_1}$ and $\tau_2$ the automorphism induced in $L_2$. Then
\[ \Ber(\tau)=\Ber (\tau_1)\cdot\Ber(\tau_2).\]
\end{lem}
\begin{proof} By Proposition \ref{additivity-2}, we have that $\omega_{X/\LL}(\tau)=\omega_{X/\LL_1} (\tau_1)\cdot \omega_{X/\LL_2}(\tau_2)$ and we conclude by Lemma \ref{berziniano=dualizante2}. 

Let us give the details of the latter equality. We follow the notations of the proof of Theorem \ref{berziniano=dualizante}. Recall that the isomorphism $\omega_{X/\LL}\to \omega_{X/\LL_1}\otimes\omega_{X/\LL_2}$ is given by the isomorphisms (of transitivity and flat base change)
\[ \omega_{X/\LL}\to s_2^*\omega_{\LL_2/\LL }\otimes\omega_{X/\LL_2}\quad \text{ and} \quad \omega_{X/\LL_1}(\pi)\colon \pi_2^*\omega_{X/\LL_1}\to \omega_{\LL_2/\LL}\,.\] By (1) and (2) of Remarks \ref{remarks-transitivity} we obtain, via the isomorphism $\omega_{X/\LL}\to s_2^*\omega_{\LL_2/\LL }\otimes\omega_{X/\LL_2}$, that $\omega_{X/\LL}(\tau)=s_2^*(\omega_{\LL_2/\LL}(\tau))\otimes \omega_{X/\LL_2}(\tau_2)$. To conclude is enough to show that, via the isomorphism $\omega_{X/\LL_1}(\pi)$, one has $\omega_{X/\LL_1}(\tau_1)=s_2^*(\omega_{\LL_2/\LL}(\tau)).$ This follows by applying $s_2^*$  to the equalities $\omega_{\LL_2/\LL}(\tau)\circ \tau_2^*(\omega_{X/\LL_1}(\pi))=\omega_{X/\LL_1}(\pi\circ\tau)=\omega_{X/\LL_1}(\tau_1\circ \pi)=\omega_{X/\LL_1}(\pi)\circ \pi_2^*(\omega_{X/\LL_1}(\tau_1))$ (see (3) of Remarks \ref{remarks-transitivity}).
\end{proof}

\begin{prop}\label{prop:Berezin}
Let $L=A^p\oplus\Pi A^q$ be a free $A$-module of rank $p\vert q$ and let $\tau\colon L\to L$ be an automorphism, which is represented  by  a matrix $$\begin{pmatrix} X & Y\\ Z&T\end{pmatrix}$$ where $X$ (resp. $T$) is a $p\times p$ (resp. $q\times q$) matrix with entries in $A_0$ and $Y$ (resp. $Z$) is a $p\times q$ (resp. $q\times p$) matrix with entries in $A_1$. Then $X$ and $T$ are invertible and
\[\aligned \Ber(\tau)&= \det(X-Y  T^{-1}  Z)\cdot \det(T)^{-1} \\ &=\det(X)\cdot\det(T-ZX^{-1}Y)^{-1}.\endaligned\]
\end{prop}

\begin{proof} One has that $\det(X)$ and $\det(T)$ are units; indeed, if $\overline A=A/J$, with $J=<A_1>$, the automorphism $\overline \tau\colon \overline L\to\overline L$ (with $\overline L=L\otimes_A\overline A$) induced by $\tau$ has the matrix $$\begin{pmatrix} \overline X & 0\\ 0&\overline T\end{pmatrix}$$ and then $\det(\overline X)$ and $\det(\overline T)$ are units (of $\overline A$). Since the elements in $J$ are nilpotent one concludes, because an element $a\in A$ is a unit if and only if its image in $\overline A$ is a unit.

Now, let us prove the equality $\Ber(\tau) = \det(X-Y  T^{-1}  Z)\cdot \det(T)^{-1}$. Let $\phi\colon L\to L$ be an automorphism whose associated matrix is $\begin{pmatrix} 1 & 0 \\ M & 1 \end{pmatrix}$ and such that 
\[ \begin{pmatrix} X & Y \\ Z & T \end{pmatrix}\cdot \begin{pmatrix} 1 & 0 \\ M & 1 \end{pmatrix}=\begin{pmatrix} X' & Y \\ 0 & T  \end{pmatrix}.\] That is, take $M=- T^{-1}  Z$, and then $X'=X-Y  T^{-1}  Z$. By Proposition \ref{berezinian-compute}  and Lemma \ref{lema-berziniano} we obtain $\Ber(\tau\circ \phi)=\det(X')\cdot \det(T)^{-1}$ and $\Ber(\phi)=\det(Id)\cdot\det(Id)^{-1}$. Therefore  
$$\det(X-Y T^{-1}  Z)\cdot \det(T)^{-1}=\Ber(\tau\circ \phi)=\Ber(\tau)\cdot\Ber(\phi)=\Ber(\tau)\cdot 1=\Ber(\tau)\,.$$
The equality $\Ber(\tau)=\det(X)\cdot\det(T-ZX^{-1}Y)^{-1}$ is proved in a similar way.
\end{proof}

\section{The super Bott formula}\label{sec:Bott}
The Bott formula allows to compute the dimension of $H^i(\bbP^n,\Omega^p_{\bbP^n}(r))$ on the complex projective space $\bbP^n$, where $\Omega^p_{\bbP^n}(r)=,\Omega^p_{\bbP^n}\otimes_{\cO_\bbP} \cO_\bbP(r)$, by means of combinatorial numbers depending on $n,p$ and $r$. One has (\cite{Bott57}\cite{OSS80}):
$$
\dim_\bbC H^i(\bbP^n,\Omega^p_{\bbP^n}(r))=
\begin{cases}
\binom{r-1}{p}\binom{r+n-p}{r}&\text{ for } i=0, 0\leq p\leq n, p<r\\
1& \text{ for }r=0, 0\leq p=i\leq n\\
\binom{-r-1}{n-p}\binom{-r+p}{-r}& \text{ for } i=n, 0\leq p\leq n, r<p-n\\\
0& \text{ otherwise} 
\end{cases}
$$
A relative version of the Bott formula is given by these authors in \cite{ASS16},   replacing the projective space by the projective bundle associated to a vector bundle on a scheme $\pi \colon\bbP(\cE)\to X$ and the cohomology groups $H^i(\bbP_k^n,\Omega^p_{\bbP^n_k}(r))$ by the higher direct images $R^i\pi_*\Omega^p_{\bbP(\cE)}(r)$. The Koszul complex of the vector bundle $\cE$ plays a crucial role in that computation since (roughly speaking) its homogeneous localization is a $\pi_*$-acyclic resolution of 
$\Omega^p_{\bbP(\cE)}(r)$ whose image by $\pi$ is the Koszul complex of $\cE$. Following this focus, an analog to the Bott formula for projective superspaces will be provided in this section. Before we begin, a remark is now in order. As we have already shown, the super Koszul complex is not in general acyclic. This is the reason  why the arguments given in \cite{ASS16} do not extend to the supercommutative case in an obvious fashion and some additional work is required.

\subsection{Cohomology of projective superspaces}
To begin with, we compute the cohomology groups of line bundles on the projective superspaces. Of course this is well-known (see for instance \cite{BR2023} or \cite{CN18}), nevertheless we include it for the convenience for further computations. Notice further that our point of view for computing these groups is slightly different from those considered so far since all of them are obtained at the same time as homogeneous components of certain $\bbZ$-graded module. 

Let $A$ be a supercommutative ring and $L=A^{m+1|n}$. Let us denote 
$$B=S^\bullet_A(L)=A[x_0,\cdots,x_{m},\theta_1,\cdots,\theta_n]$$ the polynomial supercommutative algebra over $A$ with $x_i$ even variables and $\theta_i$ odd variables. Given $\bbP^{m|n}_A=\SProj_A B$ and $\bbA^{m+1|n}=\SSpec B$, let $\sigma_0\colon \SSpec\, A\hookrightarrow \bbA^{m+1|n}$ the zero section, $U=\bbA^{m+1|n}-\{0\}$ the corresponding complementary open subset and $$\pi\colon \bbA^{m+1|n}-\{0\}\to \bbP^{m|n}$$ the quotient map (by the action of the multiplicative group). For every even variable $x_i$, put $U_{x_i}^h=\bbP_{\bar{A}}^m-\big(x_i\big)^+_0$ and $U_{x_i}=\SSpec \bar{A}[x_0\dots,x_m]-(x_i)_0$\,.

We also denote by $\cO_{\bbP^{m|n}}(r)$ the line bundle of rank $(1,0)$ defined as $\cO_{\bbP^{m|n}}(r)=\widetilde{B[r]}^h$, for any $r\in \bbZ$. Unless confusion can arise, we shall simply denote  $\bbP=\bbP^{m|n}$ and $\cO_{\bbP}(r)=\cO_{\bbP^{m|n}}(r)$.

\begin{lem}\label{prop:Ucoh}
Let $\cO_U$ be the restriction of $\widetilde{B}$ to $U$. Then, $\pi_*\cO_U=\displaystyle\bigoplus_{r\in\bbZ}\cO_{\bbP}(r)$ and $$H^i(U,\cO_U)=\displaystyle\bigoplus_{r\in\bbZ}H^i\big(\bbP,\cO_{\bbP}(r)\big)\,.$$
Consequently, $H^i\big(\bbP,\cO_{\bbP}(r)\big)$ is equal to the $r$-th homogeneous component of $H^i(U,\cO_U)$. 
\end{lem}
\begin{proof}
By using the notations mentioned earlier we have
$$
\aligned
\pi_*\cO_U(U_{x_i}^h)&=\cO_U(\pi^{-1}(U_{x_i}^h))=\cO_U(U_{x_i})=\widetilde{B}(U_{x_i})=B_{x_i}=\bigoplus_{r\in\bbZ}\big[B_{x_i}\big]_r=\\
&=\bigoplus_{r\in\bbZ}\Big[B[r]_{x_i}\Big]_0=\bigoplus_{r\in\bbZ}\widetilde{B[r]}^h(U_{x_i}^h)=\displaystyle\bigoplus_{r\in\bbZ}\cO_{\bbP^{m|n}}(r)(U_{x_i}^h)\,.
\endaligned
$$
Since $\pi$ is an affine morphism it follows that 
$$H^i(U,\cO_U)=H^i(\bbP,\pi_*\cO_U)=\displaystyle\bigoplus_{r\in\bbZ}H^i\big(\bbP,\cO_{\bbP}(r)\big)\,.$$
\end{proof}

For the case  $m=0$, a straightforward calculation shows that the projective $(0,n)$-superspace $\bbP^{0|n}$ is isomorphic to $\SSpec A[\theta_1,\dots,\theta_n]$ and $\cO_{\bbP^{0|n}}(r)\simeq \cO_{\bbA^{0|n}}$, for all $r\in\bbZ$. Hence we shall only have non-trivial the $0$-th cohomolo\-gy group for any $n$ and 
$$
\aligned H^0(\bbP^{0|n},\cO_{\bbP^{0|n}}(r))&\simeq x_0^r\cdot A[\frac{\theta_1}{x_0}\dots,\frac{\theta_n}{x_0}]\simeq\displaystyle\bigoplus_{j=0}^nx^{r-j}A[\theta_1\dots,\theta_n]_j\simeq\\&=\Big[A[x]_x\otimes_AA[\theta_1\dots,\theta_n]\Big]_r\simeq\\&=\Big[H^0(\bbA^1-(x)_0,\cO_{\bbA^1-(x)_0})\otimes_AA[\theta_1\dots,\theta_n]\Big]_r\,.
\endaligned
$$

For $m>0$, the open subset $U=\bbA^{m+1|n}-\{0\}$ is no longer affine, however the computation of $H^i\big(\bbP,\cO_{\bbP}(r)\big)$ can be accomplished via the following procedure involving the local cohomology exact sequence. 

We first note that since $\theta_i$ is nilpotent for $1\leq i\leq n$ the open subset $U$ can be written as $U'\times_{\SSpec A} \bbA^{0|n}$ where $U'=\bbA^{m+1}-\{0\}$ denotes the complementary open subset of the zero section $\SSpec A\hookrightarrow \bbA^{m+1}$. Moreover $\bbA^{0|n}\to\SSpec A$ is affine and flat. Thus, by base change one has 
\begin{equation}\label{eq:n=0}
H^i(U,\cO_U)\simeq H^i(U',\cO_{U'})\otimes_A A[\theta_1,\dots,\theta_n]
\end{equation}
so that we may assume without loss of generality that $n=0$.

Next, since $H^i(\bbA^{m+1},\cO_{\bbA^{m+1}})$ vanishing for $i\neq 0$, from the local cohomology exact sequence
$$\cdots\to H^i_{\{0\}}(\bbA^{m+1},\cO_{\bbA^{m+1}})\to H^i(\bbA^{m+1},\cO_{\bbA^{m+1}})\to H^i(U,\cO_U)\to\cdots $$
we get the exact sequence
\begin{equation}\label{eq:seqlocalcohaffine}
0\to H^0_{\{0\}}(\bbA^{m+1},\cO_{\bbA^{m+1}})\to A[x_0\dots,x_m]\to H^0(U,\cO_U)\to H^1_{\{0\}}(\bbA^{m+1},\cO_{\bbA^{m+1}})\to 0
\end{equation}
and isomorphisms 
\begin{equation}\label{eq:localcohaffine}
H^i(U,\cO_U)\isom H^{i+1}_{\{0\}}(\bbA^{m+1},\cO_{\bbA^{m+1}})
\end{equation}
for all $i\geq 1$. 

Finally, $\bbA^{m+1}=\bbA^1\times_{\SSpec A}\bbA^m$ so the computation of $H^{i}_{\{0\}}(\bbA^{m+1},\cO_{\bbA^{m+1}})$ is reduced to the case $m=1$ once we have proved the following result.

\begin{lem}\label{lem:localcoh}
Let $X$ be a superscheme over  $\SSpec A$ and let $Y\subseteq X$ be a closed sub-superscheme, then $$H^1_{\{0\}}(\bbA^1)\otimes_A H^{i}_{Y}(X)\simeq H^i_{\{0\}\times_{A} Y}(\bbA^1\times_{A} X)$$
where, by a slight abuse of notation, we are writting $A$ instead of $\SSpec A$ in the fiber products and the cohomologies mean those of the corresponding the structural sheaves. 
\end{lem}

\begin{proof}
Since $\bbA^1-\{0\}=\SSpec A[x]_x$ is affine  the exact sequence 
$$0\to H^0_{\{0\}}(\bbA^1)\to A[x]\to A[x]_x \to H^1_{\{0\}}(\bbA^1)\to 0$$
shows that $H^i_{\{0\}}(\bbA^1)=0$  for all $i\neq 1$ and $$H^1_{\{0\}}(\bbA^1)\simeq A[x]_x/A[x]\simeq \frac{1}{x} A[\frac{1}{x}]\,.$$
In particular, the exact sequence
\begin{equation}\label{eq:split} 
0\to A[x]\to A[x]_x \to H^1_{\{0\}}(\bbA^1)\to 0
\end{equation}
splits (as $A$-supermodules) because of $H^1_{\{0\}}(\bbA^1)$ is a free $A$-supermodule.

Consider now any superscheme $X$ over $\SSpec A$. After the flat base change $\bbA^1\to\SSpec A$ one finds that
$$
\aligned
H^i(\bbA^1\times_A X)&\simeq A[x]\otimes _AH^i(X)  \\
H^i\big((\bbA^1-\{0\})\times_A X\big)&\simeq A[x]_x\otimes _AH^i(X)\,.
\endaligned
$$
and, since the exact sequence \ref{eq:split} splits, from the long exact squence of local cohomology 
$$\cdots\to H^i(\bbA^1\times_A X)\to H^i\big((\bbA^1-\{0\})\times_A X\big)\to H^{i+1}_{0\times_A X}(\bbA^1\times_A X)\to\cdots
$$
we conclude that $$H^{i+1}_{\{0\}\times_A X}(\bbA^1\times_A X)\simeq H^1_{\{0\}}(\bbA^1)\otimes_A H^i(X)\,.$$
The same argument shows that 
$$H^{i+1}_{\{0\}\times_A (X-Y)}\big(\bbA^1\times_A (X-Y)\big)\simeq H^1_{\{0\}}(\bbA^1)\otimes_A H^i(X-Y)$$
for any $Y\subseteq X$ sub-superscheme.

Twisting by $H^1_0(\bbA^1)$ (which is flat) the long exact sequence 
$$\cdots\to H^i_Y(X)\to H^i(X)\to H^i(X-Y)\to \cdots$$
we obtain the long exact sequence
$$\cdots\to H^1_0(\bbA^1)\otimes_AH^i_Y(X)\to H^1_0(\bbA^1)\otimes_AH^i(X)\to H^1_0(\bbA^1)\otimes_AH^i(X-Y)\to \cdots$$
from which follows that 
$$H^i_{\{0\}\times_{A} Y}(\bbA^1\times_{A} X)\simeq H^1_{\{0\}}(\bbA^1)\otimes_A H^{i}_{Y}(X)\,.$$

\end{proof}
\begin{prop}\label{prop:localcohaffine}
Let $\bbA^{m}=\SSpec A[x_1,\dots,x_m]$. One has  
$$
H^i_{\{0\}}(\bbA^{m},\cO_{\bbA^{m}})=\begin{cases}
0 & i\neq m\\
\frac{1}{x_1\dots x_m}A\Big[\frac{1}{x_1},\dots,\frac{1}{x_m}\Big] & i=m
\end{cases}
$$
\end{prop} 

\begin{proof}
From Lemma \ref{lem:localcoh} we have
\begin{equation}\label{eq:flatbasechange}
H^1_{\{0\}}(\bbA^1)\otimes_A H^{i-1}_{\{0\}}(\bbA^{m-1})\simeq H^i_{\{0\}}(\bbA^m)\,,
\end{equation}
so that the result follows proceeding by induction on $m$. 
\end{proof}

\begin{thm}\label{thm:HO(r)}
Let $B(m,n)=A[x_0\dots x_m,\theta_1\dots\theta_n]$ and $\bbP_A^{m|n}=\SProj_A B(m,n)$. For any $r\in\bbZ$ one has:
\begin{enumerate}
\item If $m=0$, then $H^i(\bbP_A^{0|n},\cO_{\bbP_A^{0|n}}(r))=\begin{cases}
\displaystyle\bigoplus_{j=0}^nx^{r-j}A[\theta_1\dots,\theta_n]_j & \text{ for } i=0\\
0 &\text{ otherwise}
\end{cases}$
\item If $m>0$, then 
$$H^i(\bbP_A^{m|n},\cO_{\bbP_A^{m|n}}(r))=
\begin{cases}
B_r & \text{ if } i=0\\
\Big[\frac{1}{x_0\cdots x_m}A[\frac{1}{x_0},\dots,\frac{1}{x_m}]\otimes_A A[\theta_1\dots,\theta_n]\Big]_r & \text{ if } i=m\\
0 &\text{ if } i\neq 0,m
\end{cases}
$$ 
\end{enumerate}
Hence, $H^0(\bbP_A^{m|n},\cO_{\bbP_A^{m|n}}(r))$ is a free $A$-supermodule of rank $(l_0,l_1)$ with
$$l_0=\displaystyle{\sum_{\substack{0\leqslant i\leqslant n\\ |i|=0}}}\binom{m+r-i}{m}\binom{n}{i}\quad \text{ and }\quad l_1=\displaystyle{\sum_{\substack{0\leqslant i\leqslant n\\ |i|=1}}}\binom{m+r-i}{m}\binom{n}{i}$$
and $H^m(\bbP_A^{m|n},\cO_{\bbP_A^{m|n}}(r))$ is a free $A$-supermodule of rank $(k_0,k_1)$ with
$$k_0=\displaystyle{\sum_{\substack{0\leqslant i\leqslant n\\ |i|=0}}}\binom{i-r-1}{m}\binom{n}{i}\quad \text{ and }\quad k_1=\displaystyle{\sum_{\substack{0\leqslant i\leqslant n\\ |i|=1}}}\binom{i-r-1}{m}\binom{n}{i}\,.$$
(Note: we are adopting the following conventions: $\binom{a}{m}=0$ whenever $a<m\neq 0$ (in particular, for $a<0$) and $\binom{a}{0}=1$ for any $a$.)
\end{thm}
\begin{proof}
It only remains to prove the case $m>0$ which is straightforwardly verified from Lemma \ref{prop:Ucoh}, Equations \ref{eq:n=0}, \ref{eq:seqlocalcohaffine}, \ref{eq:localcohaffine} and Proposition \ref{prop:localcohaffine}.
\end{proof}

\subsection{Cohomology of \texorpdfstring{$\Omega^p_{\bbP^{m|n}/A}(r)$: }{} Bott's super formula}

For any sheaf of $\cO_\bbP$-supermodules $\cN$ we denote by  
$\cN(r)$ the sheaf $\cN\otimes_{\cO_\bbP} \cO_\bbP(r)$ and by $\widetilde{N}$ the sheaf of $\cO_\bbP$-supermodules obtained from the graded $B$-supermodule $N$ by homogeneous localization (for simplicity we omit the superscript $h$ used early). As taking homogeneous localization commutes with exterior powers and $\widetilde{(N\otimes_A B[r])}=\rho^*N(r)$ for any $A$-supermodule $N$ (note that by $N$ we denote both an $A$-supermodule and its associated sheaf on $\SSpec A$) we have 
$$\widetilde\Omega^{p}_{B/A}:={\bigwedge}^p_{\cO_\bbP}\wt\Omega_{B/A}=\wt{{\bigwedge}^p_B\Omega_{B/A}}=\big(\rho^*{\bigwedge}^p_AA^{m+1|n}\big)(-p)\,,$$
where $\rho\colon \bbP=\bbP_A^{m|n}\to \SSpec A$ is the natural morphism.

Taking homogeneous localization, the super Koszul complex $\SKos_A(A^{m+1|n})$ (see Definition \ref{def:KoszulDeRham}) lifts to a complex of $\cO_\bbP$-supermodules which we denote by $\wt{\SKos_A}(A^{m+1|n})$:
\begin{equation}\label{eq:Koszultilde}
\xymatrix@C=17pt{
\cdots \ar[r] & \widetilde\Omega^{d}_{B/A}\ar[r]^(.5){i_D}&\widetilde\Omega^{d-1}_{B/A}\ar[r]^(.6){i_D}&\cdots\ar[r]^(.34){i_D}&\widetilde\Omega_{B/A}\ar[r]^(.55){i_D}&\cO_\bbP\ar[r]  & 0
}
\end{equation}
The differentials $i_D\colon\widetilde\Omega^{d}_{B/A}\to \widetilde\Omega^{d}_{B/A}$ are morphism of $\cO_{\bbP}$-supermodules, then twistting by $\cO_\bbP(r)$ we obtain, for any $r\in\bbZ$,  a complex $\wt{\SKos_A}(A^{m+1|n})(r)$ whose differential is still denoted $i_D$.

\begin{prop}\label{prop:Koszultilde}
The complex $\wt{\SKos_A}(A^{m+1|n})$ is acyclic (that is, an exact sequence) and $$\Omega_{\bbP/A}^p=\Ker\Big(\widetilde\Omega^{p}_{B/A}\xrightarrow{i_D} \widetilde\Omega^{p-1}_{B/A}\Big)\,.$$
Then, there are exact sequences
$$0\to\Omega^p_{\bbP/A}\xrightarrow{}\widetilde{\Omega}^p_{B/A}\xrightarrow{}\Omega_{\bbP/A}^{p-1}\to 0 $$
and right and left resolutions of $\Omega^p_{\bbP/A}$:
$$0\to\Omega^p_{{\bbP}/A}\xrightarrow{} \widetilde{\Omega}^p_{B/A}\xrightarrow{} \widetilde{\Omega}^{p-1}_{B/A}\xrightarrow{}\cdots\xrightarrow{} \widetilde{\Omega}_{B/A}\xrightarrow{}  \cO_{\bbP}\to 0$$
$$\cdots \to\widetilde{\Omega}^{m+1}_{B/A}\xrightarrow{}\cdots\xrightarrow{} \widetilde{\Omega}^{p+1}_{B/A}\xrightarrow{} \Omega^p_{{\bbP}/A}\xrightarrow{} 0\,.$$
In particular (for $p=1$), we get the (super) Euler sequence 
\begin{equation}\label{eq:SuperEulerSeq}
0\to\Omega_{\bbP/A}\to \cO^{m+1|n}_{\bbP}(-1)\to\cO_{\bbP}\to 0
\end{equation}
and we conclude that 

$$\Ber(\bbP_A^{m|n}):=\Ber(\Omega_{{\bbP}/A})\isom\cO_{\bbP}(n-m-1)\,.$$
\end{prop}

\begin{proof}
It follows from slightly modifying the proof for the purely even case that we gave in (\cite[Theorem 1.4]{ASS16}).

First, we observe that even though $A^{m+1|n}\otimes_A B[-1]\to B$ is not surjective, it is surjective for each positive degree, so that we get, after taking homogeneous localization, the exact sequence 
$$0\to \Ker i_D\to \widetilde{\Omega}_{B/A}\xrightarrow{i_D}  \cO_{\bbP}\to 0$$
which splits locally since $\cO_{\bbP}$ is free. Then, it induces exact sequences
$$
0\to{\bigwedge}_{\cO_{\bbP}}^k\Ker i_D\to \widetilde{\Omega}^k_{B/A}\xrightarrow{}  {\bigwedge}_{\cO_{\bbP}}^{k-1}\Ker i_D\to 0
$$
and joining all of them one obtains the acyclicity of $\wt{\SKos_A}(A^{m+1|n})$.

Now we shall prove that $\Ker i_D\simeq \Omega_{\bbP/X}$. First we show that there exists an inyective morphism $\varphi\colon \Omega_{\bbP/X}\to \widetilde\Omega_{B/A}$ whose composition with $\widetilde\Omega_{B/A}\xrightarrow{i_D} \cO_{\bbP}$ vanishes. For any even degree 1 element $x_i\in B$, let us  denote by  $U_i$ the basic open subset of $\bbP^{m|n}_A$ corresponding to the affine superspace $U^h_{x_i}=\Spec [B_{x_i}]_0$. For brevity, we will denote $[B_{x_i}]_0$ (the zero $\bbZ$-degree component of $B_{x_i}$) by $B_{(x_i)}$. 
An homogeneous basis of $\Omega_{B/A}$ is given by $\{dx_0,\dots, dx_m, d\theta_1,\dots, d\theta_n\}$ (whose elements have degree 1, $|dx_i|=0$ and $|d\theta_k|=1$). 
By the Theorem \ref{thm:relativediff}, this basis defines a basis (in degree $1$) of the graded free $B$-supermodule $B^{m+1|n}[-1]$ which we referred to as $\{e_0,\dots, e_m,\eta_1,\dots,\eta_n\}$. Then,  the natural inclusion $B_{(x_i)}\hookrightarrow B_{x_i}$ induces a natural (degree 0) homogeneous morphism of graded supermodules 
$$\Gamma(U_i, \Omega_{\bbP/A})= \Omega_{B_{(x_i)}/A}\to \Big(\Omega_{B/A}\Big)_{(x_i)}=\Gamma(U_i,\widetilde{\Omega}_{B/A})$$ by letting
$$d(x_j/x_i)\mapsto \frac{e_jx_i-x_je_i}{x_i^2}\quad (\text{for } i\neq j)\quad \text{ and }\quad 
d(\theta_k/x_i)\mapsto \frac{\eta_kx_i-\theta_ke_i}{x_i^2}\,.$$
A straightforward computation shows that on the intersections $U_i\cap U_j$ the corresponding maps agree, thus these morphisms glue and we get a global morphism $\varphi\colon \Omega_{\bbP/X}\to \widetilde\Omega_{B/A}$. Since $B_{(x_i)}\hookrightarrow B_{x_i}$ has a retract, given by $\frac{p_d(x_j,\theta_k)}{x_i^k}\mapsto\frac{p_d(x_j,\theta_k)}{x_i^d}$, it induces one at differentials level and therefore the morphism $\varphi\colon \Omega_{\bbP/X}\to \widetilde\Omega_{B/A}$ is inyective. Furthermore, the composition $\Omega_{\bbP/X}\xrightarrow{\varphi} \widetilde\Omega_{B/A}\xrightarrow{i_D}\cO_{\bbP}$ is null because in each $U_i$ since we have
$$
(i_D\circ \varphi)( d(x_j/x_i))=i_D\left(\frac{e_jx_i - x_je_i}{x_i^{2}}\right)=\frac{i_D(e_j)x_i-  i_D(e_i)x_j}{x_i^{2}} =0
$$
and
$$
(i_D\circ \varphi)( d(\theta_k/x_i))=i_D\left(\frac{\eta_k x_i - \theta_ke_i}{x_i^{2}}\right)=\frac{i_D(\eta_k)x_i-  i_D(e_i)\theta_k}{x_i^{2}} =0
$$
because $i_D(e_j)=x_j,\, i_D(e_i)=x_i$ and $i_D(\eta_k)=\theta_k$. Hence, $\Omega_{\bbP/A}\subseteq \Ker i_D$.

Finally, as $\ker i_D\subseteq \Img(\widetilde\Omega^2_{B/A}\xrightarrow{i_D}\widetilde\Omega_{B/A})$ we conclude by showing that the image is also contained in $\Omega_{\bbP/A}$, for which it is enough to see that in each $U_i$ the following equalities hold:
$$
\begin{aligned}
i_D\left(\frac{e_j\wedge e_k}{x_i^2}\right)&=\frac{x_j}{x_i}d\left(\frac{x_k}{x_i}\right)-\frac{x_k}{x_i}d\left(\frac{x_j}{x_i}\right)
\\
i_D\left(\frac{\eta_j\wedge \eta_k}{x_i^2}\right)&=-\frac{\theta_j}{x_i}d\left(\frac{\theta_k}{x_i}\right)-\frac{\theta_k}{x_i}d\left(\frac{\theta_j}{x_i}\right)\\
i_D\left(\frac{\eta_j\wedge e_k}{x_i^2}\right)&=\frac{\theta_j}{x_i}d\left(\frac{x_k}{x_i}\right)+\frac{x_k}{x_i}d\left(\frac{\theta_j}{x_i}\right)
\end{aligned}
$$
where all right members belong to $\Omega_{B_{(x_i)}/A}$.
\end{proof}

As in the ordinary setting, the differentials $d\colon \Omega^p_{B/A}\to \Omega^{p+1}_{B/A}$ are compatible with homogeneous localization. Hence, for any $r\in\bbZ$, we have morphisms of sheaves (which also denote by $d$)
$$d\colon \wt\Omega^p_{B/A}(r)\to \wt\Omega^{p+1}_{B/A}(r)$$ such that $d\circ i_D+i_D\circ d=r\, Id$ on $\wt\Omega^p_{B/A}(r)$. It should be noticed that the morphisms $d$ are $A$-linear but not $\cO_{\bbP}$-linear, therefore the complex $(\wt\Omega^{\bullet}_{B/A}(r),d)$ is not the complex obtained for $r=0$ twisted by $\cO_{\bbP}(r)$. This yields the following splitting result.

\begin{cor}\label{cor:Qsplit}
If $A$ is a $\bbQ$-algebra, then for any $r\neq 0$ the exact sequences
$$0\to \Omega^p_{\bbP/A} (r)\to \wt\Omega^p_{B/A}(r)\to \Omega^{p-1}_{\bbP/A} (r)\to 0$$ split as sheaves of $A$-modules (but not as $\cO_\bbP$-modules).
\end{cor} 

From Theorem \ref{thm:HO(r)} and by using the projection formula we obtain that when dealing with $\widetilde\Omega^p_{B/A}(r)$ we shall only have non trivial $0$-th and $m$-th cohomology groups for any integer number $r$, more precisely:
\begin{cor}\label{cor:cohhomdif}
\hspace{2em}
\begin{enumerate}
\item If $m=0$, then 
$$
H^i(\bbP^{0|n},\widetilde\Omega^p_{B/A}(r))=
\begin{cases}
{\bigwedge}^p_A A^{1|n}\otimes_A x_0^{r-p}A[\frac{\theta_1}{x_0}\dots,\frac{\theta_n}{x_0}]& \text{ if } i=0\\
0 &\text{ otherwise } 
\end{cases}
$$
\item If $m>0$, then
\begin{equation*}
H^i(\bbP^{m|n},\widetilde\Omega^p_{B/A}(r))=
\begin{cases}
{\bigwedge}^p_A A^{m+1|n}\otimes_AB_{r-p} & \text{ if } i=0\\
{\bigwedge}^p_AA^{m+1|n}\otimes_A \Big[\frac{1}{x_0\cdots x_m}A[\frac{1}{x_0},\dots,\frac{1}{x_m}]\otimes_A A[\theta_1\dots,\theta_n]\Big]_{r-p} & \text{ if } i=m\\
0 &\text{ if } i\neq 0,m
\end{cases}
\end{equation*}
\end{enumerate}
\end{cor}

Before approaching the problem of computing the cohomology modules $H^i(\bbP^{m|n},\Omega^p_{{\bbP}/A}(r))$, let us point out some remarkable differences between super context and the ordinary one. In the first place, $\SKos_A(A^{m+1|n})$  is not a below bounded complex,  so that there are non trivial cases to be analyzed when the power $p$ takes values bigger that m. Second,  since $$\Big[\frac{1}{x_0\cdots x_m}A[\frac{1}{x_0},\dots,\frac{1}{x_m}]\otimes_A A[\theta_1\dots,\theta_n]\Big]_{q}=0\iff -m-1+n<q\,,$$ it follows that $\widetilde\Omega^p_{B/A}(r)$ is acyclic if and only if $-m-1+n<r-p$.  As a consequence, even assuming $r>0$, and unlike what happens in the ordinary case, the  resolution 
\begin{equation}\label{eq:rightresolution}
0\to\Omega^p_{{\bbP}/A}(r)\xrightarrow{} \widetilde{\Omega}^p_{B/A}(r)\xrightarrow{} \widetilde{\Omega}^{p-1}_{B/A}(r)\xrightarrow{}\cdots\xrightarrow{} \widetilde{\Omega}_{B/A}(r)\xrightarrow{}  \cO_{\bbP}(r)\to 0
\end{equation}
is not necessarily $\Gamma$-acyclic, so in general it might not be useful for our purpose. 

On the other hand, rather than computing each $H^i(\bbP^{m|n},\Omega^p_{{\bbP}/A}(r))$ independently, we may address the problem from the perspective in previous section and calculating all at once, that is, considering the affine map 
$$\pi\colon U=\bbA^{m+1|n}-\{0\}\to \bbP^{m|n}$$ and 
obtaining $H^i(\bbP^{m|n},\Omega^p_{{\bbP}/A}(r))$ as the $r$-th homogeneous component of 
\begin{equation}\label{eq:HOmega}
H^i(U,\pi^*\Omega^p_{{\bbP}/A})=\bigoplus_{r\in\bbZ}H^i(\bbP^{m|n},\Omega^p_{{\bbP}/A}(r))\,.
\end{equation}

Let us set $X=\bbA^{m+1|n}$ and, as usual, given a (graded) $B$-module $M$ we denote by $\widetilde{M}$ the (graded) associated sheaf to $M$ on $X$ (note that we use the same symbol to denote both associated sheaves and homogenous localizations). The $\bbZ$-grading on $\widetilde{M}$ is that locally induced by the $\bbZ$-grading on $M$, that is, for each  basic open subset $U_{x_i}$ $$\widetilde{M}_{| U_{x_i}}=\displaystyle{\bigoplus_{r\in\bbZ} \widetilde{[M_{x_i}]_r}}\,.$$
In this notation, $\widetilde{\SKos_A L}$ is the complex of sheaves on $X$
$$
\xymatrix{
\Omega^\bullet_X\equiv  \cdots\ar[r] & \Omega^2_X\ar[r] & \Omega_X\ar[r] & \cO_X\ar[r] & 0
}
$$
where each $\Omega^q_X$ is, as $\bbZ$-graded sheaf on $X$, isomorphic to $\bigwedge^q_A L\otimes_A \cO_X[-q]$. Moreover, the cokernel of $\Omega_X\to \cO_X$ is supported on the zero section, therefore the complex $\Omega^\bullet_X$ is not acyclic. However, restricting to $U=X-\{0\}$ we get (by an analogous argument than that in Proposition \ref{prop:Koszultilde}) an acyclic complex
$$
\xymatrix{
\Omega^\bullet_U\equiv  \cdots\ar[r] & \Omega^2_U\ar[r] & \Omega_U\ar[r] & \cO_U\ar[r] & 0
}
$$
such that for every $p\geq 0$
\begin{equation}\label{eq:resolutionU}
0\to\pi^*\Omega^p_{{\bbP}/A}\xrightarrow{} {\Omega}^p_{U}\xrightarrow{} {\Omega}^{p-1}_{U}\xrightarrow{}\cdots\xrightarrow{} {\Omega}_{U}\xrightarrow{}  \cO_{U}\to 0
\end{equation}
is a right resolution of $\pi^*\Omega^p_{{\bbP}/A}$. As $\bbZ$-graded sheaf, each $\Omega^q_U$ is isomorphic to $\bigwedge^q_A L\otimes_A \cO_U[-q]$ and the $r$-th homogeneous component of $\pi_*\cO_U[-q]$ is $\cO_{\bbP}(r-q)$. It follows that
$$\pi_*\Omega_U^q\simeq  {\exto}^p_A L\otimes_A\cO_\bbP(-q) \otimes_{\cO_\bbP} \pi_*\cO_U\simeq \widetilde{\Omega}^q_{B/A}\otimes_{\cO_\bbP}\pi_*\cO_U$$
and 
\begin{equation}\label{eq:HOmegatilde}
H^i(U,\Omega_U^q)\simeq \bigoplus_{r\in\bbZ} H^i(\bbP,  \widetilde{\Omega}^{q}_{B/A}(r))\,.
\end{equation}
Hence, Corollary \ref{cor:cohhomdif} yields $H^i(U,\Omega_U^q)=0$ for all $i\neq 0,m$.

Assume $m>0$ and consider, for any $p\geq 0$, the exact triangle
$$
\bbR\Gamma_0(X,\Omega^p_X)\to\bbR\Gamma(X,\Omega^p_X)\to\bbR\Gamma(U,\Omega_U^p)
$$
Taking cohomology and by using that 
\begin{equation}\label{eq:gammazero}
 H^i_{\{0\}}(X,\Omega_X^p)\simeq{\bigwedge}^p_A A^{m+1|n}\otimes_A H^i_{\{0\}}(X,\cO_X)=0 \text{ for any } i\neq m+1
\end{equation} 
we obtain for any $p\geq 0$
\begin{equation}\label{eq:HOmegaU}
H^i(U,\Omega^p_U)=\begin{cases}
\Gamma(X,\Omega_X^p) & \text{ if } i=0\\ 
H^{m+1}_{\{0\}}(X,\Omega_X^p)& \text{ if } i=m\\
0&\text{ otherwise}
\end{cases} 
\end{equation}

Therefore $\Gamma(X,\Omega_X^\bullet)\simeq \Gamma(U,\Omega^\bullet_U)$ and the $i$-th cohomology of the complex

$$\Gamma\big(U,R^{\bullet}\big)\equiv\xymatrix@C=17pt{
0\ar[r] &\Gamma\big(U,{\Omega}^{p}_{U}\big)\ar[r]^(.65){i_D}&\dots\ar[r]^(.35){i_D} & \Gamma\big(U,{\Omega}_U\big)\ar[r]^(.5){i_D} &\Gamma(U,\cO_U)\ar[r] & 0
}
$$
is (assuming $m>0$) $$H^i\big(\Gamma(U,R^{\bullet})\big)=\begin{cases}
Z^{-p}(\SKos_A L) & \text{ if } i=0\\
H^{i-p}\big(\SKos_A (L)\big) &\text{ otherwise }
\end{cases}
$$
where $Z^{q}(\SKos_A (L))$ denotes the ${q}$-cycles of the Koszul complex $\SKos_A(L)$, that is 
$$Z^{-q}(\SKos_A (L)):=\Ker\Big(\Omega_{B/A}^q\xrightarrow{i_D}{\Omega_{B/A}^{q-1}}\Big)\,.$$

We also denote by $B^{-p}\big(H^m(U,\Omega_U^\bullet)\big)$ the $(-p)$-boundaries of the complex $H^m(U, \Omega_U^\bullet)$, that is,
$$B^{-p}\big(H^m(U,\Omega_U^\bullet)\big)=\Img\Big(H^m(U,\Omega_U^{p+1})\to H^m(U,\Omega_U^{p})\Big)\,.$$

\begin{prop}\label{prop:Bott}
Assume $m>0$ and consider the right resolution of  $\pi^\ast\Omega^p_{\bbP/A}$ given by  
$$0\to \pi^\ast\Omega^p_{\bbP/A}\to {\Omega}^p_{U}\xrightarrow{} {\Omega}^{p-1}_{U}\xrightarrow{}\cdots\xrightarrow{} {\Omega}_{U}\xrightarrow{}  \cO_{U}\to 0$$ Then, $\Gamma(U,\pi^*\Omega_\bbP^p)\simeq Z^{-p}(\SKos_AL)$, there exist isomorphisms
$$H^{i-p}(\SKos_AL)\isom H^i(U,\pi^*\Omega_\bbP^p) \quad\text{ for all }0< i\leq m-1$$
and the following sequence
$$ 
\xymatrix{
0\ar[r]& H^{m-p}\big(\SKos_A(L)\big)\ar[r] & H^m(U,\pi^*\Omega^p_\bbP)\ar[r] &{\displaystyle\bigoplus_{r\leq p-m-1+n} B^{-p}\big(H^{m}(U,\Omega_U^{\bullet})_r\big)}\ar[r] & 0
}
$$
is exact.
\end{prop}

\begin{proof}
Let us denote by$\pi^\ast\Omega^p_{\bbP/A}\to R^\bullet$ the right resolution of  $\pi^\ast\Omega^p_{\bbP/A}$ given in the statement. Although $\pi^*\Omega^p_{{\bbP}/A}\to R^\bullet$ is not a resolution by acyclic sheaves, applying the functor of global sections $\Gamma(X,\--)$ one has, by de Rham's theorem in its derived version, the morphism (in the derived category)
$$\Gamma(U,R^\bullet)\to \bbR\Gamma(U,R^\bullet)\simeq \R\Gamma(U,\pi^*\Omega^p_{{\bbP}/A})\,.$$
This morphism can be embedded into an exact triangle computing its cone as follows. For any $R^{q}$, choose a resolution by acyclic sheaves, say $R^{q}\to (K^{q\bullet},\delta^{q\bullet})$. Taking global sections we have a sequence
$$0\to\Gamma(U,R^{q})\xrightarrow{}\Gamma(U,K^{q\bullet})$$ which is exact on each term $\Gamma(U,K^{q,j})$ for $j\neq m$. Denoting by  $\Gamma(U, K^{q\bullet})_{\leq t}$ the truncated complex
$$\Gamma(U,K^{q,0})\to \cdots\to \Gamma(U,K^{q,t-2})\to \Gamma(U,K^{q,t-1})\to\ker \Gamma(\delta^{q,t})\to0$$ we have that the natural morphism
$$\Gamma(U,R^{q})\xrightarrow{} \Gamma(U,K^{q\bullet})_{\leq m-1}$$
is a quasi-isomorphism and 
$$
\Gamma(U, R^{q})\to \Gamma(U,K^{q\bullet})_{\leq m}\to H^m(U, R^{q})[-m]$$
is an exact triangle. The simplex complex associated to $K^{\bullet\bullet}$ gives a resolution which allows to compute $\bbR\Gamma(U,R^{\bullet})$ and we obtain the exact triangle
\begin{equation}\label{eq:triangle}
\Gamma(U,R^{\bullet})\to \bbR\Gamma(U,\pi^*\Omega^p_{{\bbP}/A}) \to H^m(U,R^{\bullet})[-m]
\end{equation}
Taking cohomology we get the isomorphisms
$$H^i(\Gamma(U,R^{\bullet}))\isom H^i(U,\pi^*\Omega_\bbP^p) \quad\text{ for all }0\leq  i\leq m-1$$
and the exact sequence
$$ 
0\to H^{m}\big(\Gamma(U, R^{\bullet})\big)\to H^m(U,\pi^*\Omega^p_\bbP)\to H^0\big(H^{m}(U,R^{\bullet})\big)\to H^{m+1}(\Gamma(U,R^\bullet))\to 0\,.
$$
where the $i$-th cohomology of the complex $\Gamma\big(U,R^{\bullet}\big)$ is  $$H^i\big(\Gamma(U,R^{\bullet})\big)=\begin{cases}
Z^{-p}(\SKos_A L) & \text{ if } i=0\\
H^{i-p}\big(\SKos_A (L)\big) &\text{ otherwise }
\end{cases}
$$

It remains to compute the kernel of the morphism $H^0\big(H^{m}(U,R^{\bullet})\big)\to H^{m+1}(\Gamma(U,R^\bullet))$. We shall proceed by using local cohomology as follows: consider the exact triangle
$$
\xymatrix{
\bbR\Gamma_0(X,\Omega_X^\bullet)\ar[r] &\bbR\Gamma(X,\Omega_X^\bullet)\ar[r] &\bbR\Gamma(U,\Omega_U^\bullet)}.
$$
The term $\bbR\Gamma(U,\Omega_U^\bullet)$ vanishes because $\Omega_U^\bullet$ is acyclic, and furthermore,  $\bbR\Gamma(X,\Omega_X^\bullet)\simeq \Gamma(X,\Omega_X^\bullet)$ because $X$ is affine. Then, 
$$\bbR\Gamma_0(X,\Omega_X^\bullet)\simeq \Gamma(X,\Omega_X^\bullet)\simeq \SKos_A(L)\,.$$
On the other hand, by Equation \ref{eq:HOmegaU} one has the isomorphisms 
$$\bbR\Gamma_0(X,\Omega_X^\bullet)\simeq H^{m+1}_{\{0\}}(X,\Omega_X^\bullet)[-m-1]\simeq H^m(U,\Omega_U^\bullet)[-m-1]\,.$$ 
Thus, the exact triangle above gives the isomorphism
$$
H^m(U,\Omega_U^\bullet)\simeq \big(\SKos_AL\big)[m+1]\,.
$$
Finally, since $H^0(H^m(U,R^\bullet))=Z^0\big(H^m(U,\Omega_U^\bullet)[-p]\big)=Z^{-p}\big(H^m(U,\Omega_U^\bullet)\big)$, 
there exists the exact sequence
$$0\to B^{-p}\big(H^m(U,\Omega_U^\bullet)\big)\to Z^{-p}\big(H^m(U,\Omega_U^\bullet)\big)\to H^{m+1}\big(\Gamma(U,R^\bullet)\big)\to 0$$

Moreover, $H^m(U,\Omega_U^p)=0$ whenever $-m-1+n<r-p$, so that we find the exact sequence
$$ 
\xymatrix{
0\ar[r]& H^{m-p}\big(\SKos_A(L)\big)\ar[r] & H^m(U,\pi^*\Omega^p_\bbP)\ar[r] &{\displaystyle\bigoplus_{r\leq p-m-1+n} B^{-p}\big(H^{m}(U,\Omega_U^{\bullet})_r\big)}\ar[r] & 0
}
$$
and we conclude.
\end{proof}
Once the cohomology modules $H^m(U,\pi^*\Omega_\bbP^p)$ are given by means of $\SKos_A(L)$ and $H^m(U,\Omega_U^\bullet)$ the modules $H^i(\bbP^{m|n},\Omega^p_{{\bbP}/A}(r))$ arise simply taking the $r$-th homogeneous components. Notice that in order to compute the $m$-th cohomolgy, we shall have to deal with two cases: when $-m-1+n<r-p$ and therefore only $\SKos_A(L)$ contributes and when $-m-1+n\geq r-p$ and therefore $\SKos_A(L)$ and $H^m(U,\Omega_U^\bullet)$ are contributing.
 
\begin{thm}\label{thm:Bott}
Let $L=A^{m+1|n},\,B(m,n)=A[x_0\dots x_m,\theta_1\dots\theta_m]$ and $\bbP_A^{m|n}=\SProj_A B(m,n)$. 
\begin{enumerate}
\item[(a)] If $m=0$, then $$H^i(\bbP^{0|n},\Omega^p_{\bbP/A}(r))=\begin{cases}K_{p,r} &\text{ if } i=0\\
0&\text{ otherwise }
\end{cases}$$
where $K_{p,r}:=\Ker\big(\ext_A^{p}A^{1|n}\otimes_A x_0^{r-p}A[\frac{\theta_1}{x_0},\dots \frac{\theta_n}{x_0}]\to \ext_A^{p-1}A^{1|n}\otimes_A x_0^{r-p+1}A[\frac{\theta_1}{x_0},\dots \frac{\theta_n}{x_0}]\big)$.
\item [(b)] If $m>0$, then one has:
\begin{enumerate}
\renewcommand{\labelenumii}{(\emph{\arabic{enumii}})}
\item For $i=0$
$$
H^0(\bbP^{m|n},\Omega^p_{\bbP/A}(r))=
\begin{cases}
Z^{-p}(\SKos_A(L)_r) & \text{ if } r>0\\
A &\text{ if } r=0\text{ and } p=0\\
0 &\text{ otherwise } 
\end{cases}
$$
\item For $0<i<m$
$$
H^i(\bbP^{m|n},\Omega^p_{\bbP/A}(r))=
\begin{cases}
H^{i-p}(\SKos_A(L)_r) & \text{ if } r>0\\
A &\text{ if } r=0\text{ and } p=i\\
0 &\text{ otherwise } 
\end{cases}
$$
\item For $i=m$ and $-m-1+n<r-p$

$$
H^m(\bbP^{m|n},\Omega^p_{\bbP/A}(r))=
\begin{cases}
H^{m-p}(\SKos_A(L)_r) & \text{ if } r>0 \\
A &\text{ if } r=0  \text{ and } p=m\\
0  & \text{ otherwise }
\end{cases}
$$
\item For $i=m$ and $-m-1+n\geq r-p$
\begin{itemize}
\item $H^m(\bbP^{m|n},\Omega^p_{\bbP/A}(r))=B^{-p}(H^m(U,\Omega_U^\bullet)_r)$ if one of the following four cases holds: $r>0$ and $m-p>0$ or $r>0$ and $m-p<-r$ or $r=0$ and $m\neq p$ or $r<0$.
\item If $r\geq0$ and $m-p\leq 0$ and $r+m-p\geq 0$ there exists the following exact sequence
$$
\xymatrix{
0\ar[r]& H^{m-p}\big(\SKos_A(L)_r\big)\ar[r] & H^m(\bbP^{m|n},\Omega^p_{\bbP/A}(r))\ar[r] & B^{-p}\big(H^{m}(U,\Omega_U^{\bullet})_r\big)\ar[r] & 0
}
$$
which for the case $r=0$ and $p=m$ turns out to be
$$
\xymatrix{
0\ar[r]& A\ar[r] & H^m(\bbP^{m|n},\Omega^m_{\bbP/A})\ar[r] & B^{-m}\big(H^{m}(U,\Omega_U^{\bullet})_0\big)\ar[r] & 0
}
$$
\end{itemize}
\end{enumerate}
\end{enumerate}
\end{thm}

\begin{proof}
\begin{enumerate}\renewcommand{\labelenumi}{(\emph{\alph{enumi}})}
\item For $m=0$, $U=\SSpec(A[x]_x\otimes_A A[\theta_1,\dots,\theta_n])$ is affine. By Equation \ref{eq:HOmega}, $H^i(\bbP^{0|n},\Omega^p_{\bbP/A}(r))=0$ for $i>0$. By Corollary \ref{cor:cohhomdif} and Equation \ref{eq:HOmegatilde}, 
$$
0\to\pi^*\Omega^p_{{\bbP}/A}\xrightarrow{} {\Omega}^p_{U}\xrightarrow{} {\Omega}^{p-1}_{U}\xrightarrow{}\cdots\xrightarrow{} {\Omega}_{U}\xrightarrow{}  \cO_{U}\to 0
$$
is a resolution of $\pi^*\Omega^p_{{\bbP}/A}$ by acyclic sheaves and 
$$
\aligned
H^0(\bbP^{0|n},\Omega^p_{\bbP/A}(r))&=H^0(U,\pi^*\Omega^p_{\bbP/A})_r=\\
&=\Ker\big(\ext_A^{p}A^{1|n}\otimes_A x_0^{r-p}A[\frac{\theta_1}{x_0},\dots \frac{\theta_n}{x_0}]\to \ext_A^{p-1}A^{1|n}\otimes_A x_0^{r-p+1}A[\frac{\theta_1}{x_0},\dots \frac{\theta_n}{x_0}]\big)\,.
\endaligned
$$

\item Follows from Proposition \ref{prop:Bott} and taking into account the following facts:

\begin{itemize}
\item $H^m(U,\pi^\ast\Omega^p_\bbP)_r\simeq B^{-p}\Big(H^m(U,\Omega_U^\bullet)_r\Big)$ whenever some of the following cases holds:
\begin{enumerate}\renewcommand{\labelenumi}{(\emph{\arabic{enumi}})}
\item $r<0$, because $\SKos_A(L)$ is a $\bbZ$-graded complex in non-negative degrees.
\item $m-p>0$, because $\SKos_A(L)$ is a complex in  non-positive terms.
\item $r-p+m<0$, because $\hSKos{-p+m}_A(L)_r=0$.
\end{enumerate}
\item On the other side, $H^m(U,\pi^\ast\Omega_{\bbP/A}^p)_r\simeq H^{m-p}(\SKos_A(L)_r)$ for $-m-1+n<r-p$ because, as we have already mentioned, $H^m(U,R^\bullet)=0$. 
\end{itemize}
\end{enumerate}
\end{proof} 

Assuming $m>0$, it follows from Theorem \ref{thm:Bott} that $H^i(\bbP^{m|n},\Omega^p_{\bbP/A})=H^{i-p}(\SSpec A,\tilde{A})$ for all $i\in[0,m)$. The equality still holds for $i=m$, that is, $H^m(\bbP^{m|n},\Omega^p_{\bbP/A})=H^{m-p}(\SSpec A,\tilde{A})$ provided that $p<m+1-n$. On the other hand, for  $p\geq m+1-n$ there are the following cases:
\begin{itemize}
\item If $p=m$, there exists the exact sequence
$$0\to A\to H^m(\bbP^{m|n},\Omega^m_{\bbP/A})\to Z^{-m}(H^m(U,\Omega_U^\bullet)_0)\to 0$$
\item If $p=m+1$, there exists the exact sequence
$$0\to H^m(\bbP^{m|n},\Omega^{m+1}_{\bbP/A})\to Z^{-m-1}(H^m(U,\Omega_U^\bullet)_0)\to A\to 0$$
\item If $p\neq m,m+1$, then $H^m(\bbP^{m|n},\Omega^{p}_{\bbP/A})\simeq Z^{-p}(H^m(U,\Omega_U^\bullet)_0)$.
\end{itemize}

By simplicity, let us denote 
$$
\aligned
\ell_{r,0}^+&=\displaystyle{\sum_{\substack{0\leqslant i\leqslant n\\ |i|=0}}}\binom{m+r-i}{m}\binom{n}{i}\quad&
\ell_{r,0}^-&=\displaystyle{\sum_{\substack{0\leqslant i\leqslant n\\ |i|=1}}}\binom{m+r-i}{m}\binom{n}{i}\\
\ell_{r,m}^+&=\displaystyle{\sum_{\substack{0\leqslant i\leqslant n\\ |i|=0}}}\binom{i-r-1}{m}\binom{n}{i}\quad&
\ell_{r,m}^-&=\displaystyle{\sum_{\substack{0\leqslant i\leqslant n\\ |i|=1}}}\binom{i-r-1}{m}\binom{n}{i}\\
\lambda_{p}^+&=\displaystyle{\sum_{\substack{0\leqslant i\leqslant p\\ |i|=0}}}\binom{m+1}{p-i}\binom{n+i-1}{i}\quad&\lambda_{p}^-&=\displaystyle{\sum_{\substack{0\leqslant i\leqslant p\\ |i|=1}}}\binom{m+1}{p-i}\binom{n+i-1}{i}
\endaligned
$$
Then, for the twisted case we have the following:
\begin{cor}[Bott's super formula]\label{cor:Bott}
Let $A$ be a $\bbQ$-algebra, then for any nonzero integer number $r$ one has:
$$
H^i(\bbP^{m|n}, \Omega^p_{\bbP^{m|n}/A}(r))=
\begin{cases}
K_{p,r}& \text{ if } i=m=0 \\
Z^{-p}(\SKos_A(L)_r) & \text{ if } i=0\\
Z^{-p}(H^m(U,\Omega_U^\bullet)_{r}) & \text{ if } i=m \\
0 & \text{ otherwise }
\end{cases}
$$
Furthermore, if $A$ is a (super) field of characteristic zero, then $H^0(\bbP^{m|n}, \Omega^p_{\bbP^{m|n}/A}(r))$ is a super vector space of dimension $(\delta_{}(p,r)_0^+,\delta_{}(p,r)_0^-)$ with
$$
\aligned
\delta_{}(p,r)_0^+&=\sum_{j=0}^p (-1)^j\big(\lambda_{p-j}^+\ell_{r-p+j,0}^++\lambda_{p-j}^-\ell_{r-p+j,0}^-\big)\\
\delta_{}(p,r)_0^-&=\sum_{j=0}^p (-1)^j\big(\lambda_{p-j}^+\ell_{r-p+j,0}^-+\lambda_{p-j}^-\ell_{r-p+j,0}^+\big)
\endaligned
$$
and $H^m(\bbP^{m|n}, \Omega^p_{\bbP^{m|n}/A}(r))$ is a super vector space of dimension $(\delta_{}(p,r)_m^+,\delta_{}(p,r)_m^-)$ with
$$
\aligned
\delta_{}(p,r)_m^+&=\sum_{j=0}^p (-1)^j\big(\lambda_{p-j}^+\ell_{r-p+j,m}^++\lambda_{p-j}^-\ell_{r-p+j,m}^-\big)\\
\delta_{}(p,r)_m^-&=\sum_{j=0}^p (-1)^j\big(\lambda_{p-j}^+\ell_{r-p+j,m}^-+\lambda_{p-j}^-\ell_{r-p+j,m}^+\big)
\endaligned
$$
\end{cor}
\begin{proof}
Since $A$ is a $\bbQ$-algebra, $\SKos_A(L)_r$ is acyclic by Theorem \ref{thm:aciclicity}. Then, the first part follows from Theorem \ref{thm:Bott}.

By Corollary \ref{cor:Qsplit} one has 
$$H^i(\bbP^{m|n}, \Omega^p_{\bbP^{m|n}/A}(r))\oplus H^i(\bbP^{m|n}, \Omega^{p-1}_{\bbP^{m|n}/A}(r))=H^i(\bbP^{m|n}, \tilde\Omega^p_{B/A}(r))$$
and we apply Theorem \ref{thm:HO(r)} and Corollary \ref{cor:cohhomdif} to conclude.
\end{proof}


\end{document}